\documentclass[a4paper, 12pt, reqno]{amsart}
\usepackage{mathrsfs}
\usepackage{amsfonts}
\usepackage[centertags]{amsmath}
\usepackage{amssymb}
\usepackage{amsthm}
\usepackage{graphicx}
\usepackage{caption}
\usepackage{hyperref}
\usepackage{enumerate}
\usepackage[textwidth=16cm, hmarginratio=1:1]{geometry}
\usepackage{appendix}

\newtheorem{theorem}{Theorem}[section]

\newtheorem{lemma}[theorem]{Lemma}
\newtheorem{proposition}[theorem]{Proposition}
\newtheorem{definition}[theorem]{Definition}

\newtheorem{conjecture}[theorem]{Conjecture}

\numberwithin{equation}{section}

\DeclareMathOperator*{\rep}{Rep}
\DeclareMathOperator*{\diag}{diag}
\DeclareMathOperator*{\res}{Res}

\begin{document}

\title{On the extended T-system of type $C_3$}
\author{Jian-Rong Li}
\address{School of mathematics and statistics, Lanzhou University, Lanzhou 730000, P. R. China.}
\email{lijr07@gmail.com, lijr@lzu.edu.cn}
%\date{}

\maketitle

\begin{abstract}
We continue the study of extended T-systems of quantum affine algebras. We find a sub-system of the extended T-system of the quantum affine algebra $U_q \hat{\mathfrak{g}}$ of type $C_3$. The sub-system consisting of four systems which are denoted by I, II, III, and IV. Each of the systems I, II, III, IV is closed. The systems I-IV can be used to compute minimal affinizations with weights of the form $\lambda_1 \omega_1 + \lambda_2 \omega_2 + \lambda_3 \omega_3$, where at least one of $\lambda_1$, $\lambda_2$, $\lambda_3$ are zero. Using the systems I-IV, we compute the characters of the restrictions of the minimal affinizations in the systems to $ U_q \mathfrak{g} $ and obtain some conjectural decomposition formulas for the restrictions of some minimal affinizations.
\end{abstract}

Mathematics Subject Classification (2010): Primary 17B37; Secondary 81R50, 82B23

\section{Introduction}
The T-systems are some families of relations in the Grothendieck ring of the category of the finite-dimensional modules of quantum affine algebras (or Yangians), see \cite{K83}, \cite{K84}, \cite{K87}, \cite{KR90}, \cite{KNS94}, \cite{Nak03}, \cite{Her06}. The T-systems are widely applied to representation theory, combinatorics and integrable systems, see the recent survey \cite{KNS11}.

Usual T-systems involve the so-called The Kirillov-Reshetikhin modules, introduced in \cite{KR90}. They are, perhaps, the best understood class of simple modules over quantum affine algebras. Recently, the usual T-systems have been generalized to the so called extended T-systems, see \cite{MY12b}, \cite{LM12}. An extended T-system involves minimal affinizations of quantum affine algebras. The family of minimal affinizations is an important family of irreducible modules which contains the Kirillov-Reshetikhin modules, see \cite{C95}, \cite{CP95b}, \cite{CP96a}, \cite{CP96b}. The minimal affinizations are also interesting from the physical point of view, see Remark 4.2 of \cite{FR92} and \cite{C95}.

Minimal affinizations are studied intensively in recently years, see for example, \cite{CMY12}, \cite{CG11}, \cite{Her07}, \cite{LM12}, \cite{Mou10}, \cite{MF11}, \cite{MY12a}, \cite{MY12b}, \cite{MY12c}, \cite{Nao12}.
The finite dimensional representations of $U_q \hat{\mathfrak{g}}$ and cluster algebras are closely related, see \cite{IIKKN13a}, \cite{IIKKN13b}, \cite{HL10}, \cite{HL13}, \cite{Nak11}. We expect that the relations in the extended T-systems are special relations in the cluster algebras. A cluster algebra algorithm for computing $q$-characters of Kirillov-Reshetikhin modules for any untwisted quantum affine algebra is given in \cite{HL13}. Their method uses the T-systems for Kirillov-Reshetikhin modules. We expect that the extended T-systems can be used to find new algorithms to compute $q$-characters of more general modules including minimal affinizations.

The extended T-systems of type $A$, $B$ has been found in \cite{MY12b} and the extended T-system of type $G_2$ has been found in \cite{LM12}. In \cite{MY12b}, it was conjectured that the extended T-systems exist in all types.

In this paper, we continue the study of extended T-systems of quantum affine algebras. Let $\rep( U_q\hat{\mathfrak{g}} )$ denote the Grothendieck ring of the category of finite-dimensional representations of $U_q\hat{\mathfrak{g}}$. The irreducible finite-dimensional modules of quantum affine algebras are parameterized by the $l$-highest weights or Drinfeld polynomials.

Our method is similar as the method used in \cite{Nak03}, \cite{Her06}, \cite{MY12b}, \cite{LM12}. Let $a\in\mathbb{C}^{\times}$ and let $\mathcal{T}$ be an irreducible $U_q\hat{\mathfrak{g}}$-module such that the zeros of the Drinfeld polynomials of $\mathcal{T}$
belong to $aq^\mathbb{Z}$. Following \cite{MY12b}, we define the left, right, and bottom modules, denoted by $\mathcal{L}$, $\mathcal{R}$, $\mathcal{B}$ respectively, as the simple modules whose Drinfeld polynomials has zeros obtained by dropping the rightmost, leftmost, and both left- and rightmost zeros of the union of zeros of the Drinfeld polynomials of the top module $\mathcal{T}$.

The $q$-character theory and the Frenkel-Mukhin (FM) algorithm are important tools to study the representation theory of quantum affine algebras, see \cite{NT98}, \cite{FR98}, \cite{FM01}. The main tools in this paper are the $q$-character theory and the FM algorithm.

Recall (\cite{Nak04}, Definition 10.1) that a $U_q\hat{\mathfrak{g}}$-module is called special if its $q$-character contains only one dominant monomial. Let $\mathfrak{g}$ be the simple Lie algebra of type $C_3$. The FM algorithm applies to special modules. Although in general the minimal affinizations of type $C_3$ are not special, there are some families of minimal affinizations of type $C$ which are special.

We use $\mathcal{T}$ to denote the $l$-highest weight module with an $l$-highest weight $T$. Let $k, \ell, m \in \mathbb{Z}_{\geq 0}$, $s\in \mathbb{Z}$. We consider families of special modules $\mathcal{T}_{k, \ell, m}^{(s)}$ and $\mathcal{\tilde{T}}_{k, \ell, m}^{(s)}$ (their $l$-highest weights are presented in Section 3.2) which are minimal affinizations of type $C_3$.

We take minimal affinizations $\mathcal{T}_{k, \ell, 0}^{(s)}$, $\mathcal{\tilde{T}}_{k, \ell, 0}^{(s)}$ as top modules $\mathcal{T}$ respectively. It turns out that the left, right, bottom modules of $\mathcal{T}$ satisfy a relation $[\mathcal{L}] [\mathcal{R}]=[\mathcal{T}] [\mathcal{B}] + [\mathcal{S}]$. Here $[ \cdot ]$ denotes the equivalence class of a $U_q\hat{\mathfrak{g}}$-module in $\rep( U_q\hat{\mathfrak{g}} )$ and $S$ is a tensor product of some irreducible modules. The factors of $\mathcal{S}$ are called sources. The factors of the sources can be some modules which are not minimal affinizations. In order to obtain a closed system, we take the sources as top modules and compute new left, right, bottom modules, and sources. We continue this procedure until all modules in the sources are the modules obtained before. Then we obtain the desired closed systems I and III.

The system I contains minimal affinizations $\mathcal{T}_{k, \ell, 0}^{(s)}$, $\mathcal{T}_{k, 0, m}^{(s)}$, and $\mathcal{\tilde{T}}_{k, 0, m}^{(s)}$ for all $k$, $\ell$, $m \in \mathbb{Z}_{\geq 0}$, $s\in \mathbb{Z}$. The system III contains minimal affinizations $\mathcal{\tilde{T}}_{k, \ell, 0}^{(s)}$ for all $k, \ell \in \mathbb{Z}$. We denote by II, IV the dual systems of I, III respectively. The system II contains minimal affinizations which can be obtained from $\mathcal{\tilde{T}}_{0, k, \ell}^{(s)}$, $k$, $\ell$, $m \in \mathbb{Z}_{\geq 0}$, $s\in \mathbb{Z}$, by shifting the upper-subscripts. The system IV contains minimal affinizations which can be obtained from $T_{0, k, \ell}^{(s)}$, $k, \ell \in \mathbb{Z}$, $s\in \mathbb{Z}$, by shifting the upper-subscripts. Therefore we find a sub-system of the extended T-system of the quantum affine algebra $U_q \hat{\mathfrak{g}}$ of type $C_3$ which is the union of the systems I, II, III, IV and the sub-system can be used to compute minimal affinizations with weights of the form $\lambda_1 \omega_1 + \lambda_2 \omega_2 + \lambda_3 \omega_3$, where at least one of $\lambda_1$, $\lambda_2$, $\lambda_3$ are zero, see Section \ref{main results}.

We find that the modules in the systems I, III are
\begin{align*}
& \mathcal{T}_{k, \ell, 0}^{(s)}, \ \mathcal{T}_{k, 0, m}^{(s)}, \ \mathcal{T}_{0, \ell, r}^{(s)}, \ \mathcal{\tilde{T}}_{k, 0, m}^{(s)}, \ \mathcal{S}_{k, \ell}^{(s)}, \ \mathcal{R}_{k, 2\ell, \ell}^{(s)}, \\
& \mathcal{R}_{k, 2\ell+1, \ell}^{(s)}, \ \mathcal{R}_{k, 2\ell+2, \ell}^{(s)}, \ \mathcal{U}_{k, \ell}^{(s)}, \ \mathcal{V}_{k, \ell}^{(s)}, \ \mathcal{P}_{k, \ell}^{(s)}, \ \mathcal{O}_{k, \ell}^{(s)},
\end{align*}
where $s\in \mathbb{Z}$, $k, \ell, m \in \mathbb{Z}_{\geq 0}$, $r\in \{0, 1, 2\}$.
We show that these modules are special. We show that every relation $[\mathcal{L}] [\mathcal{R}]=[\mathcal{T}] [\mathcal{B}] + [\mathcal{S}]$ in the system holds by comparing the dominant monomials in both sides of $[\mathcal{L}] [\mathcal{R}]=[\mathcal{T}] [\mathcal{B}] + [\mathcal{S}]$. Moreover, we show that the modules $\mathcal{T}\otimes \mathcal{B}$ and $\mathcal{S}$ in the systems I-IV are irreducible.

Our main results are Theorems \ref{the system I}, \ref{the system II}, \ref{the system III}, \ref{the system IV} in Section \ref{main results}.

The extended T-system is a powerful tool of studying the finite dimensional representations of $U_q \hat{\mathfrak{g}}$. Using the systems I-IV, we compute the characters of the restrictions of the minimal affinizations in the systems to $ U_q \mathfrak{g} $ and obtain some conjectural decomposition formulas for the restrictions of some minimal affinizations, see Section \ref{conjectural formulas}.

Let $\mathcal{T}$ be a $U_q\hat{\mathfrak{g}}$-module. We denote by $\res(\mathcal{T})$ the restriction of $\mathcal{T}$ to a $U_q \mathfrak{g}$-module. As a vector space, $\res(\mathcal{T})$ is the same as $\mathcal{T}$. But we consider $\res(\mathcal{T})$ as a $U_q \mathfrak{g}$-module.

The paper is organized as follows. In Section \ref{background}, we give some background material. In Section \ref{main results}, we describe the system I, II, III, and IV. In Section \ref{prove special}, we prove that the modules in the systems I and III are special. In Section \ref{prove system}, we prove Theorem \ref{the system I} and Theorem \ref{the system III}. In Section \ref{prove irreducible}, we prove that the module $\mathcal{T} \otimes \mathcal{B}$ is irreducible for each relation in the systems I, II, III, and IV. In Section \ref{prove compute}, we prove Proposition \ref{compute 1}.
In Section \ref{conjectural formulas}, we give conjectural character formulas for $\res(\mathcal{T}_{k, \ell, 0}^{(s)})$, $\res(\mathcal{T}_{k, 0, m}^{(s)})$, $\res(\mathcal{\tilde{T}}_{k, 0, m}^{(s)})$, $s\in\mathbb{Z}$, $k$, $\ell$, $m \in  \mathbb{Z}_{\geq 0}$.

\section{Background} \label{background}

\subsection{Quantum affine algebra}

Let $\mathfrak{g}$ be a complex simple Lie algebra of type $C_3$ and $\mathfrak{h}$ a Cartan subalgebra of $\mathfrak{g}$. Let $I=\{1, 2, 3\}$.
We choose simple roots $\alpha_1, \alpha_2, \alpha_3$ and scalar product $(\cdot, \cdot)$ such that
\begin{align*}
( \alpha_1, \alpha_1 ) = 2, \ ( \alpha_2, \alpha_2 )=2, \ ( \alpha_3, \alpha_3 )=4, \ ( \alpha_1, \alpha_2 )=-1, \ ( \alpha_1, \alpha_3 )= 0, \ ( \alpha_2, \alpha_3 )=-2.
\end{align*}
Let $\{\alpha_1^{\vee}, \alpha_2^{\vee}, \alpha_3^{\vee}\}$ and $\{\omega_1, \omega_2, \omega_3\}$ be the sets of simple coroots and fundamental weights respectively. Let $C=(C_{ij})_{i, j\in I}$ denote the Cartan matrix, where $C_{ij}=\frac{2 ( \alpha_i, \alpha_j ) }{( \alpha_i, \alpha_i )}$. Let $r_1=1, r_2=1, r_3=2$, $D=\diag (r_1, r_2, r_3)$ and $B=DC$.
Then
\begin{align*}
C=
\left(
\begin{array}{ccc}
2 & -1 & 0 \\
-1 & 2 & -2 \\
0 & -1 & 2
\end{array}
\right), \
B=
\left(
\begin{array}{ccc}
2 & -1 & 0 \\
-1 & 2 & -2 \\
0 & -2 & 4
\end{array}
\right).
\end{align*}

Let $Q$ (resp. $Q^+$) and $P$ (resp. $P^+$) denote the $\mathbb{Z}$-span (resp. $\mathbb{Z}_{\geq 0}$-span) of the simple roots and fundamental weights respectively. There is a partial order $\leq$ on $P$ such that $\lambda \leq \lambda'$ if and only if $\lambda' - \lambda \in Q^+$.

The quantum affine algebra $U_q\hat{\mathfrak{g}}$ is a $\mathbb{C}(q)$-algebra generated by $x_{i, n}^{\pm}$ ($i\in I, n\in \mathbb{Z}$), $k_i^{\pm 1}$ $(i\in I)$, $h_{i, n}$ ($i\in I, n\in \mathbb{Z}\backslash \{0\}$) and central elements $c^{\pm 1/2}$, subject to certain relations, see \cite{Dri88}. The algebra $U_q\hat{\mathfrak{g}}$ is a Hopf algebra.

Let $U_q\mathfrak{g}$ be the quantized enveloping algebra of $\mathfrak{g}$. The subalgebra of $U_q\hat{\mathfrak{g}}$ generated by $(k_i^{\pm})_{i\in I}, (x_{i, 0}^{\pm})_{i\in I}$ is a Hopf subalgebra of $U_q\hat{\mathfrak{g}}$ and is isomorphic as a Hopf algebra to $U_q\mathfrak{g}$. Therefore $U_q\hat{\mathfrak{g}}$-modules restrict to $U_q\mathfrak{g}$-modules. We denote by $\res(V)$ the restriction of a $U_q\hat{\mathfrak{g}}$-module $V$ to $U_q\mathfrak{g}$.

\subsection{The $q$-characters of finite-dimensional $U_q\hat{\mathfrak{g}}$-modules}
A $U_q\hat{\mathfrak{g}}$-module is called a module of type $1$ if $c^{\pm 1/2}$ acts as the identity on $V$ and
\begin{align} \label{decomposition into weight spaces}
V=\bigoplus_{\lambda \in P} V_{\lambda}, \ V_{\lambda} = \{v\in V  :   k_i v=q^{( \alpha_i, \lambda )} v \}.
\end{align}
We will only consider finite-dimensional type $1$ $U_q\hat{\mathfrak{g}}$-modules in this paper. The decomposition (\ref{decomposition into weight spaces}) of a finite-dimensional $U_q\hat{\mathfrak{g}}$-module $V$ into its $U_q\mathfrak{g}$-weight spaces can be refined as follows, see \cite{FR98}:
\begin{align}
V=\bigoplus_{\gamma} V_{\gamma}, \ \gamma=(\gamma_{i, \pm r}^{\pm})_{i\in I, r\in \mathbb{Z}_{\geq 0}}, \ \gamma_{i, \pm r}^{\pm} \in \mathbb{C},
\end{align}
where
\begin{align*}
V_{\gamma} = \{v\in V  :   \exists k\in \mathbb{N}, \forall i \in I, m \geq 0, (\phi_{i, \pm m}^{\pm} - \gamma_{i, \pm m}^{\pm})^{k} v =0 \}.
\end{align*}
Here $\phi_{i, \pm m}^{\pm}$ is defined by the following formula
\begin{align*}
\sum_{n=0}^{\infty} \phi_{i, \pm m}^{\pm} u^{\pm n} = k_i^{\pm 1} \exp\left( \pm( q-q^{-1} ) \sum_{m=1}^{\infty} h_{i, \pm m} u^{\pm m} \right). 
\end{align*}
If $\dim(V_{\gamma})>0$, then $\gamma$ is called an \textit{$l$-weight} of $V$, see Section 1.3 in \cite{Nak01}.

Let $\mathcal{P}$ denote the free Abelian multiplicative group of monomials in infinitely many formal variables $(Y_{i, a})_{i\in I, a\in \mathbb{C}^{\times}}$. In \cite{FR98} (Proposition 1 and the proof of Theorem 3), it is shown that there is a bijection $\gamma$ from from $\mathcal{P}$ to the set of $l$-weights of finite-dimensional $U_q\hat{\mathfrak{g}}$-modules. Therefore the monomials in $\mathcal{P}$ are also called $l$-weights. For simplicity, we denote $V_m=V_{\gamma(m)}$.

Let $\mathbb{Z}\mathcal{P} = \mathbb{Z}[Y_{i, a}^{\pm 1}]_{i\in I, a\in \mathbb{C}^{\times}}$. For $\chi \in \mathbb{Z}\mathcal{P}$, we write $m\in \mathcal{P}$ if the coefficient of $m$ in $\chi$ is non-zero.

The $q$-character of a $U_q\hat{\mathfrak{g}}$-module $V$ is given by
\begin{align*}
\chi_q(V) = \sum_{m\in \mathcal{P}} \dim(V_{m}) m \in \mathbb{Z}\mathcal{P},
\end{align*}
where $V_m = V_{\gamma(m)}$.

Let $\rep(U_q\hat{\mathfrak{g}})$ be the Grothendieck ring of finite-dimensional representations of $U_q\hat{\mathfrak{g}}$ and $[V]\in \rep(U_q\hat{\mathfrak{g}})$ the class of a finite-dimensional $U_q\hat{\mathfrak{g}}$-module $V$. The $q$-character map
\begin{eqnarray*}
\chi_q: \rep(U_q\hat{\mathfrak{g}}) & \to & \mathbb{Z}\mathcal{P}, \\
V & \mapsto & \chi_q(V),
\end{eqnarray*}
is an injective ring homomorphism, see \cite{FR98}, Theorem 3 (1).

Given finite-dimensional $U_q\hat{\mathfrak{g}}$-module $V$, denote $\mathscr{M}(V)$ the set of all monomials in $\chi_q(V)$. For each $j\in I$, a monomial $m=\prod_{i\in I, a\in \mathbb{C}^{\times}} Y_{i, a}^{u_{i, a}}$, where $u_{i, a}$ are some integers, is said to be \textit{$j$-dominant} (resp. \textit{$j$-anti-dominant}) if and only if $u_{j, a} \geq 0$ (resp. $u_{j, a} \leq 0$) for all $a\in \mathbb{C}^{\times}$. A monomial is called \textit{dominant} (resp. \textit{anti-dominant}) if and only if it is $j$-dominant (resp. $j$-anti-dominant) for all $j\in I$. Let $\mathcal{P}^+ \subset \mathcal{P}$ denote the set of all dominant monomials.

Let $V$ be a $U_q\hat{\mathfrak{g}}$-module and $m\in \mathscr{M}(V)$ a monomial. Let $v \in V_m$ be a non-zero vector. If
\begin{align*}
x_{i, r}^{+} \cdot v=0, \ \phi_{i, \pm t}^{\pm} \cdot v=\gamma(m)_{i, \pm t}^{\pm} v, \ \forall i\in I, r\in \mathbb{Z}, t\in \mathbb{Z}_{\geq 0},
\end{align*}
then $v$ is called an \textit{$l$-highest weight vector} with \textit{$l$-highest weight} $\gamma(m)$. The module $V$ is called an \textit{$l$-highest weight module} if $V=U_q\hat{\mathfrak{g}}\cdot v$ for some $l$-highest weight vector $v\in V$.

In \cite{CP94}, \cite{CP95a}, it is shown that there is a bijection between the set of isomorphism classes of finite-dimensional irreducible $l$-highest weight $U_q\hat{\mathfrak{g}}$-modules of type $1$. Let $L(m_+)$ denote the irreducible $l$-highest weight $U_q\hat{\mathfrak{g}}$-module corresponding to $m_+\in \mathcal{P}^{+}$. We use $\chi_q(m_+)$ to denote $\chi_q(L(m_+))$ for $m_+ \in \mathcal{P}^+$.

We have the following well-known lemma, see \cite{CP95a}, Corollary 3.5.
\begin{lemma}
\label{contains in a larger set}
Let $m_1, m_2$ be two monomials. Then $L(m_1m_2)$ is a sub-quotient of $L(m_1) \otimes L(m_2)$. In particular, $\mathscr{M}(L(m_1m_2)) \subseteq \mathscr{M}(L(m_1))\mathscr{M}(L(m_2))$.   $\Box$
\end{lemma}

For $b\in \mathbb{C}^{\times}$, we define the shift of spectral parameter map $\tau_b: \mathbb{Z}\mathcal{P} \to \mathbb{Z}\mathcal{P}$ to be a ring homomorphism sending $Y_{i, a}^{\pm 1}$ to $Y_{i, ab}^{\pm 1}$. Let $m_1, m_2 \in \mathcal{P}^+$. If $\tau_b(m_1) = m_2$, then
\begin{align} \label{shift}
\tau_b \chi_q(m_1) = \chi_q(m_2).
\end{align}

\begin{definition} [{\cite{Nak04}, \cite{Her07}}]
A finite-dimensional $U_q\hat{\mathfrak{g}}$-module $V$ is said to be \textit{special} if and only if $\mathscr{M}(V)$ contains exactly one dominant monomial. It is called \textit{anti-special} if and only if $\mathscr{M}(V)$ contains exactly one anti-dominant monomial.  
\end{definition}

Clearly, if a module is special or anti-special, then it is irreducible.

Define $A_{i, a} \in \mathcal{P}, i\in I, a\in \mathbb{C}^{\times}$, by
\begin{align*}
& A_{1, a} = Y_{1, aq}Y_{1, aq^{-1}} Y_{2, a}^{-1}, \\
& A_{2, a} = Y_{2, aq}Y_{2, aq^{-1}} Y_{1, a}^{-1} Y_{3, a}^{-1}, \\
& A_{3, a} = Y_{3, aq^{2}}Y_{3, aq^{-2}} Y_{2, aq}^{-1} Y_{2, aq^{-1}}^{-1}.
\end{align*}
Let $\mathcal{Q}$ be the subgroup of $\mathcal{P}$ generated by $A_{i, a}, i\in I, a\in \mathbb{C}^{\times}$. Let $\mathcal{Q}^{\pm}$ be the monoids generated by $A_{i, a}^{\pm 1}, i\in I, a\in \mathbb{C}^{\times}$. There is a partial order $\leq$ on $\mathcal{P}$ such that
\begin{align}
m\leq m' \text{ if and only if } m'm^{-1}\in \mathcal{Q}^{+}. \label{partial order of monomials}
\end{align}
For all $m_+ \in \mathcal{P}^+$, $\mathscr{M}(L(m_+)) \subset m_+\mathcal{Q}^{-}$, see \cite{FM01}, Theorem 4.1.

Let $m$ be a monomial. If for all $a \in \mathbb{C}^{\times}$ and $i\in I$, we have the property: if the power of $Y_{i, a}$ in $m$ is non-zero and the power of $Y_{j, aq^k}$ in $m$ is zero for all $j \in I, k\in \mathbb{Z}_{>0}$, then the power of $Y_{i, a}$ in $m$ is negative, then the monomial $m$ is called \textit{right negative}. For $i\in I, a\in \mathbb{C}^{\times}$, $A_{i,a}^{-1}$ is right-negative. A product of right-negative monomials is right-negative. If $m$ is right-negative and $m'\leq m$, then $m'$ is right-negative.

\subsection{Minimal affinizations of $U_q\mathfrak{g}$-modules}
Let $V(\mu)$ be a simple finite dimensional $U_q(\mathfrak{g})$-module of highest weight $\mu$. We say that a simple finite dimensional $U_q(\hat{\mathfrak{g}})$-module $L(m)$ is an affinization of $V(\mu)$ if $\omega(m) = \mu$, where $\omega: \mathcal{P} \to P$ is the homomorphism of abelian groups defined by $\omega(Y_{i, a}) = \omega_i$, $i \in I$, see \cite{C95}. Two affinizations are said to be equivalent if they are isomorphic as $U_q(\mathfrak{g})$-modules. 

Let $V$ be a finite dimensional $U_q(\mathfrak{g})$-module. Then $V$ can be decomposed as a direct sum of simple $U_q(\mathfrak{g})$-modules. For each $\lambda \in P^+$, let $m_{\lambda}(V)$ denote the multiplicity of $V(\lambda)$ in $V$. Let $[[L(m)]]$ denote the equivalent class of $L(m)$ and $\mathcal{Q}_{\mu}$ denote the set of equivalent classes of affinizations of $V(\mu)$. There is a partial order "$\leq $" in $\mathcal{Q}_{\mu}$ defined as follows. For $[[L(m)]], [[L(m')]] \in \mathcal{Q}_{\mu}$, $[[L(m)]] \leq [[L(m')]]$ if and only if for all $\mu \in P^+$ one of the following holds:
\begin{enumerate}[(i)]
\item $m_{\mu}(L(m)) \leq m_{\mu}(L(m'))$;

\item there is some $\nu > \mu$ such that $m_{\nu}(L(m)) < m_{\mu}(L(m'))$.
\end{enumerate}
An affinization in $\mathcal{Q}_{\mu}$ is called a minimal affinization if it is minimal with respect to the partial order "$\leq$". If $[[L(m)]]$ is a minimal affinization, then we also call the module $L(m)$ a minimal affinization.

The \textit{minimal affinizations} of type $C_3$ are classified in \cite{CP95b}. Let $\mathfrak{g}$ be the simple Lie algebra of type $C_3$ and $V(\lambda)$ the $U_q(\mathfrak{g})$ module of weight $\lambda$, where $\lambda = m_1 \omega_1 + m_2 \omega_2 + m_3 \omega_3$ and $m_1, m_2, m_3 \in \mathbb{Z}_{\geq 0}$. A simple $U_q \hat{\mathfrak{g}}$-module $L(m_+)$ is a minimal affinization of $V(\lambda)$ if and only if $m_+$ is one of the following monomials
\begin{align*}
T_{m_3, m_2, m_1}^{(s)} = \left( \prod_{i=0}^{m_3-1} Y_{3, aq^{s+4i}} \right) \left( \prod_{i=0}^{m_2-1} Y_{2, aq^{s+4k+2i+1}} \right) \left( \prod_{i=0}^{m_1-1} Y_{1, aq^{s+4k+2\ell+2i+2}} \right),
\end{align*}
\begin{align*}
\tilde{T}_{m_1, m_2, m_3}^{(s)} = \left( \prod_{i=0}^{m_1-1} Y_{1, aq^{s+2i}} \right) \left( \prod_{i=0}^{m_2-1} Y_{2, aq^{s+2k+2i+1}} \right) \left( \prod_{i=0}^{m_3-1} Y_{3, aq^{s+2k+2\ell+4i+4}} \right),
\end{align*}
for some $a\in \mathbb{C}^{\times}$, see \cite{CP95b}. In particular, when $m_1=0, m_2=0$ or $m_2=0, m_3=0$ or $m_1=0, m_3=0$, the minimal affinization $L(m_+)$ is a \textit{Kirillov-Reshetikhin module}.

Let $L(m_+)$ be a Kirillov-Reshetikhin module. It is shown in \cite{Her06} that any non-highest monomial in $\mathscr{M}(L(m_+))$ is right-negative and hence $L(m_+)$ is special.

\subsection{The $q$-characters of $U_q\hat{\mathfrak{sl}}_2$-modules and the FM algorithm}

The $q$-characters of $U_q\hat{\mathfrak{sl}}_2$-modules are well-understood, see \cite{CP91}, \cite{CP95a}, \cite{FR98}. We recall the results here.

Let $W_{k}^{(a)}$ be the irreducible representation $U_q\hat{\mathfrak{sl}}_2$ with
highest weight monomial
\begin{align*}
X_{k}^{(a)}=\prod_{i=0}^{k-1} Y_{aq^{k-2i-1}},
\end{align*}
where $Y_a=Y_{1, a}$. Then the $q$-character of $W_{k}^{(a)}$ is given by
\begin{align*}
\chi_q(W_{k}^{(a)})=X_{k}^{(a)} \sum_{i=0}^{k} \prod_{j=0}^{i-1} A_{aq^{k-2j}}^{-1},
\end{align*}
where $A_a=Y_{aq^{-1}}Y_{aq}$.

For $a\in \mathbb{C}^{\times}, k\in \mathbb{Z}_{\geq 1}$, the set $\Sigma_{k}^{(a)}=\{aq^{k-2i-1}\}_{i=0, \ldots, k-1}$ is called a \textit{string}. Two strings $\Sigma_{k}^{(a)}$ and $\Sigma_{k'}^{(a')}$ are said to be in \textit{general position} if the union $\Sigma_{k}^{(a)} \cup \Sigma_{k'}^{(a')}$ is not a string or $\Sigma_{k}^{(a)} \subset \Sigma_{k'}^{(a')}$ or $\Sigma_{k'}^{(a')} \subset \Sigma_{k}^{(a)}$.

Let $L(m_+)$ denote the irreducible $U_q\hat{\mathfrak{sl}}_2$-module with
highest weight monomial $m_+$. Let $m_{+} \neq 1$, $m_+ \in \mathbb{Z}[Y_a]_{a\in \mathbb{C}^{\times}}$, be a dominant monomial. Then $m_+$ can be uniquely written in the form
\begin{align*}
m_+=\prod_{i=1}^{s} \left( \prod_{b\in \Sigma_{k_i}^{(a_i)}} Y_{b} \right),
\end{align*}
up to a permutation, where $s$ is an integer, $\Sigma_{k_i}^{(a_i)}, i=1, \ldots, s$, are strings which are pairwise in general position and
\begin{align*}
L(m_+)=\bigotimes_{i=1}^s W_{k_i}^{(a_i)}, \qquad \chi_q(m_+)=\prod_{i=1}^s
 \chi_q(W_{k_i}^{(a_i)}).
\end{align*}

For $j\in I$, let
\begin{align*}
\beta_j : \mathbb{Z}[Y_{i, a}^{\pm 1}]_{i\in I; a\in \mathbb{C}^{\times}} \to \mathbb{Z}[Y_{a}^{\pm 1}]_{a\in \mathbb{C}^{\times}}
\end{align*}
be the ring homomorphism which sends, for all $a\in \mathbb{C}^{\times}$, $Y_{k, a} \mapsto 1$ for $k\neq j$ and $Y_{j, a} \mapsto Y_{a}$.

Let $V$ be a $U_q\hat{\mathfrak{g}}$-module. Then $\beta_i(\chi_q(V))$, $i=1, 2$, is the $q$-character of $V$ considered as a $U_{q_i}(\hat{\mathfrak{sl}_2})$-module. In Section 5 of \cite{FM01}, a powerful algorithm is proposed to compute the $q$-characters of $U_q\hat{\mathfrak{g}}$-modules. The algorithm is called the FM algorithm. The FM algorithm recursively computes the minimal possible $q$-character which contains $m_+$ and is consistent when restricted to $U_{q_i}(\hat{\mathfrak{sl}_2})$, $i \in I$. Although the FM algorithm does not produce the correct $q$-characters for some modules, it works for a large class of modules. In particular, if a module $L(m_+)$ is special, then the FM algorithm produces the correct $q$-character $\chi_q(m_+)$, see \cite{FM01}.

\subsection{Truncated $q$-characters}

In this paper, we need to use the concept truncated $q$-characters, see \cite{HL10}, \cite{MY12b}.
Given a set of monomials $\mathcal{R} \subset \mathcal{P}$, let $\mathbb{Z}\mathcal{R} \subset \mathbb{Z}\mathcal{P}$ denote the $\mathbb{Z}$-module of formal linear combinations of elements of $\mathcal{R}$ with integer coefficients. Define
\begin{align*}
\text{trunc}_{\mathcal{R}}: \mathcal{P} \to \mathcal{R}; \quad m\mapsto \begin{cases} m & \text{if } m\in \mathcal{R}, \\ 0 & \text{if } m \not\in \mathcal{R}, \end{cases}
\end{align*}
and extend $\text{trunc}_{\mathcal{R}}$ as a $\mathbb{Z}$-module map $\mathbb{Z}\mathcal{P} \to \mathbb{Z}\mathcal{R}$.

Given a subset $U\subset I \times \mathbb{C}^{\times}$, let $\mathcal{Q}_U$ be the subgroups of $\mathcal{Q}$ generated by $A_{i,a}$ with $(i,a)\in U$. Let $\mathcal{Q}_U^{\pm}$ be the monoid generated by $A_{i,a}^{\pm 1}$ with $(i,a)\in U$. The polynomial $\text{trunc}_{m_{+}\mathcal{Q}_U^{-}} \: \chi_q(m_{+})$ is called \textit{the $q$-character of $L(m_{+})$ truncated to $U$}.

The following theorem can be used to compute some truncated $q$-characters.

\begin{theorem}[{ Theorem 2.1, \cite{MY12b} }]  \label{truncated}
Let $U\subset I \times \mathbb{C}^{\times}$ and $m_{+} \in \mathcal{P}^+$. Suppose that $\mathcal{M} \subset \mathcal{P}$ is a finite set of distinct monomials such that
\begin{enumerate}[(i)]
\item $\mathcal{M} \subset m_+\mathcal{Q}_U^{-}$,
\item $\mathcal{P}^+ \cap \mathcal{M} = \{m_{+}\}$,
\item for all $m\in \mathcal{M}$ and all $(i,a)\in U$, if $mA_{i,a}^{-1} \not\in \mathcal{M}$, then $mA_{i,a}^{-1}A_{j,b} \not\in \mathcal{M}$ unless $(j,b)=(i,a)$,
\item for all $m \in \mathcal{M}$ and all $i \in I$, there exists a unique $i$-dominant monomial $M\in \mathcal{M}$ such that
\begin{align*}
\text{trunc}_{\beta_i(M\mathcal{Q}_U^{-})} \: \chi_q(\beta_i(M)) = \sum_{m'\in m\mathcal{Q}_{\{i\} \times \mathbb{C}^{\times}} \cap \mathcal{M} } \beta_i(m').
\end{align*}
\end{enumerate}
Then
\begin{align*}
\text{trunc}_{m_{+}\mathcal{Q}_U^{-}} \: \chi_q(m_+) = \sum_{m\in \mathcal{M}} m.
\end{align*}
\end{theorem}

Here $\chi_q(\beta_i(M))$ is the $q$-character of the irreducible $U_{q_i}(\hat{\mathfrak{sl}_2})$-module with highest weight monomial $\beta_i(M)$ and $\text{trunc}_{\beta_i(M\mathcal{Q}_{U}^{-})}$ is the polynomial obtained from $\chi_q(\beta_i(M))$ by keeping only the monomials of $\chi_q(\beta_i(M))$ in the set $\beta_i(M\mathcal{Q}_U^-)$.

\section{Extended T-system of type $C_3$} \label{main results}
In this section, $\mathfrak{g}$ is the simple Lie algebra of type $C_3$. We will describe the extended T-system which contains all minimal affinizations of type $C_3$ whose weights are of the form $\lambda_1 \omega_1 + \lambda_2 \omega_2 + \lambda_3 \omega_3$, where at least one of $\lambda_1, \lambda_2, \lambda_3$ are $0$. The extended T-system consisting of four sub-systems I, II, III, IV. Each of the systems I, II, III, IV is closed.

In the following, we fix $a \in \mathbb{C}^{\times}$ and use $i_s^{\pm}$ to denote $Y_{i, aq^s}^{\pm}$, $i \in I, s \in \mathbb{Z}$. 
\subsection{Fundamental $q$-characters of type $C_3$} First we use the FM algorithm to compute the $q$-characters of the fundamental modules type $C_3$.

\begin{lemma} \label{fundamental q-characters}

The the $q$-characters of the fundamental modules of type $C_3$ are
\begin{eqnarray*}
\chi_q(1_0) = 1_{0}  +  1_{2}^{-1}  2_{1}  +  2_{3}^{-1}  3_{2}  +  2_{5}  3_{6}^{-1}  +  1_{6}  2_{7}^{-1} + 1_{8}^{-1},
\end{eqnarray*}

\begin{eqnarray*}
\chi_q(2_0) & = & 2_{0}  + 1_{1}  2_{2}^{-1}  3_{1} + 1_{3}^{-1}  3_{1} +  1_{1}  2_{4}  3_{5}^{-1} +  1_{3}^{-1}  2_{2}  2_{4}  3_{5}^{-1}  + 1_{1}  1_{5}  2_{6}^{-1} \\
 & & +  1_{3}^{-1}  1_{5}  2_{2}  2_{6}^{-1} +  1_{5}  2_{4}^{-1}  2_{6}^{-1}  3_{3} +  1_{1}  1_{7}^{-1}  +  1_{3}^{-1}  1_{7}^{-1}  2_{2}  +  1_{7}^{-1}  2_{4}^{-1}  3_{3}  \\
 & & + 1_{5}  3_{7}^{-1} + 1_{7}^{-1}  2_{6}  3_{7}^{-1}  +  2_{8}^{-1},
\end{eqnarray*}

\begin{eqnarray*}
\chi_q(3_0) & = & 3_{0}  + 2_{1}  2_{3}  3_{4}^{-1} +  1_{4}  2_{1}  2_{5}^{-1} +  1_{2}  1_{4}  2_{3}^{-1}  2_{5}^{-1}  3_{2} + 1_{6}^{-1}  2_{1} +  1_{2}  1_{6}^{-1}  2_{3}^{-1}  3_{2} \\
 & & + 1_{4}^{-1}  1_{6}^{-1}  3_{2}  +  1_{2}  1_{4}  3_{6}^{-1} + 1_{2}  1_{6}^{-1}  2_{5}  3_{6}^{-1} +  1_{4}^{-1}  1_{6}^{-1}  2_{3}  2_{5}  3_{6}^{-1} +  1_{2}  2_{7}^{-1} \\
  & & +  1_{4}^{-1}  2_{3}  2_{7}^{-1}  +  2_{5}^{-1}  2_{7}^{-1}  3_{4} +  3_{8}^{-1}. \qquad \qquad \qquad \qquad \qquad \qquad \Box
\end{eqnarray*}
\end{lemma}

\subsection{The system I}
For $s\in \mathbb{Z}, k, \ell, m \in \mathbb{Z}_{\geq 0}$, we define the following monomials.
\begin{align*}
T_{k, \ell, m}^{(s)} = \left( \prod_{i=0}^{k-1} 3_{s+4i} \right) \left( \prod_{i=0}^{\ell-1} 2_{s+4k+2i+1} \right) \left( \prod_{i=0}^{m-1} 1_{s+4k+2\ell+2i+2} \right),
\end{align*}

\begin{align*}
\tilde{T}_{k, \ell, m}^{(s)} = \left( \prod_{i=0}^{k-1} 1_{s+2i} \right) \left( \prod_{i=0}^{\ell-1} 2_{s+2k+2i+1} \right) \left( \prod_{i=0}^{m-1} 3_{s+2k+2\ell+4i+4} \right),
\end{align*}

\begin{align*}
S_{k, \ell}^{(s)} = \left( \prod_{i=0}^{k-1} 2_{s+2i} \right) \left( \prod_{i=0}^{\ell-1} 2_{s+2k+2i+4} \right),
\end{align*}

\begin{align*}
R_{k, \ell, m}^{(s)} = \left( \prod_{i=0}^{k-1} 2_{s+2i} \right) \left( \prod_{i=0}^{\ell-1} 1_{s+2k+2i+1} \right) \left( \prod_{i=0}^{m-1} 3_{s+2k+4i+3} \right).
\end{align*}

Let $T$ be a monomial and $\mathcal{T}$ the highest $l$-weight module of $U_q \hat{\mathfrak{g}}$ with highest $l$-weight $T$. For $k, \ell \in \mathbb{Z}_{\geq 0}, s\in \mathbb{Z}$, we have the following trivial relations:
\begin{equation}
\begin{aligned}
& \mathcal{\tilde{T}}_{k, 0, 0}^{(s)} = \mathcal{T}_{0, 0, k}^{(s-2)}, \ \mathcal{\tilde{T}}_{0, k, 0}^{(s)} = \mathcal{T}_{0, k, 0}^{(s)}, \ \mathcal{\tilde{T}}_{0, 0, k}^{(s)} = \mathcal{T}_{k, 0, 0}^{(s+4)}, \\
& \mathcal{S}_{k, 0}^{(s)} = \mathcal{T}_{0, k, 0}^{(s-1)}, \ \mathcal{S}_{0, k}^{(s)} = \mathcal{T}_{0, k, 0}^{(s+4)}, \\
& \mathcal{R}_{ k , \ell, 0}^{(s)} = \mathcal{T}_{ 0, k , \ell}^{(s-1)}, \ \mathcal{R}_{k, 0, 0}^{(s)} = \mathcal{T}_{0, k, 0}^{(s)}, \ \mathcal{R}_{0, k, 0}^{(s)} = \mathcal{T}_{0, 0, k}^{(s-1)}, \ \mathcal{R}_{0, 0, k}^{(s)} = \mathcal{T}_{k, 0, 0}^{(s+3)}. \label{trivial relations 1}
\end{aligned}
\end{equation}

\begin{theorem} \label{special}
The modules $\mathcal{T}_{k, \ell, 0}^{(s)}$, $\mathcal{T}_{k, 0, m}^{(s)}$, $\mathcal{T}_{0, \ell, r}^{(s)}$, $\mathcal{\tilde{T}}_{k, 0, m}^{(s)}$, $\mathcal{S}_{k, \ell}^{(s)}$, $\mathcal{R}_{k, 2\ell, \ell}^{(s)}$, $\mathcal{R}_{k, 2\ell+1, \ell}^{(s)}$, $\mathcal{R}_{k, 2\ell+2, \ell}^{(s)}$ are special. In particular, we can use the FM algorithm to compute these modules.
\end{theorem}

We will prove Theorem \ref{special} in Section \ref{prove special}.

For $k \in \mathbb{Z}_{\geq 1}$, we have the following relations in $\rep(U_q \hat{\mathfrak{g}})$ in the usual T-system, see \cite{KR90}, \cite{Her06},
\begin{align}
[ \mathcal{T}_{k-1, 0, 0}^{(s)} ] [ \mathcal{T}_{k-1, 0, 0}^{(s+4)} ] =[ \mathcal{T}_{k, 0, 0}^{(s)} ] [ \mathcal{T}_{k-2, 0, 0}^{(s+4)} ] + [ \mathcal{T}_{0, 2k-2,  0}^{s+1}], \label{usual T equation 1}
\end{align}

\begin{align}
[ \mathcal{T}_{0, k-1, 0}^{(s)} ] [ \mathcal{T}_{0, k-1, 0}^{(s+2)} ] =[ \mathcal{T}_{0, k, 0}^{(s)} ] [ \mathcal{T}_{0, k-2, 0}^{(s+2)} ] + [ \mathcal{T}_{k-1, 0,  0}^{s+1}][\mathcal{T}_{0, 0,  \lfloor \frac{k}{2} \rfloor}^{s+1}][\mathcal{T}_{ 0, 0,  \lfloor \frac{k-1}{2} \rfloor }^{s+3}], \label{usual T equation 2}
\end{align}

\begin{align}
[ \mathcal{T}_{k-1, 0, 0}^{(s)} ] [ \mathcal{T}_{k-1, 0, 0}^{(s+2)} ] =[ \mathcal{T}_{k, 0, 0}^{(s)} ] [ \mathcal{T}_{k-2, 0, 0}^{(s+2)} ] + [ \mathcal{T}_{0, k-1,  0}^{s+1}]. \label{usual T equation 3}
\end{align}

Let $\sigma: \mathbb{Z}_{\geq 0} \to \{0, 1\}$ be the function such that $\sigma(k)=0$ if $k$ is even and $\sigma(k)=1$ if $k$ is odd.

Our first main result is the following theorem.
\begin{theorem} \label{the system I}
For $k, \ell, m \in \mathbb{Z}_{\geq 1}$, $r \in \{1, 2\}$, we have the following relations in $\rep(U_q \hat{\mathfrak{g}})$.
\begin{align}
[ \mathcal{T}_{k, \ell-1, 0}^{(s)} ] [ \mathcal{T}_{k-1, \ell, 0}^{(s+4)} ] =[ \mathcal{T}_{k, \ell, 0}^{(s)} ] [ \mathcal{T}_{k-1, \ell-1, 0}^{(s+4)} ] + [ \mathcal{R}_{2k-1, \ell,  \lfloor \frac{\ell}{2} \rfloor}^{(s+1)} ] [ \mathcal{T}_{\lfloor \frac{\ell-1}{2} \rfloor, 0, 0}^{(s+4k+4)} ], \label{relation 1 in the system I}
\end{align}

\begin{align}
[ \mathcal{T}_{k, 0, m-1}^{(s)} ] [ \mathcal{T}_{k-1, 0, m}^{(s+4)} ] =[ \mathcal{T}_{k, 0, m}^{(s)} ] [ \mathcal{T}_{k-1, 0, m-1}^{(s+4)} ] + [ \mathcal{S}_{2k-1, m-1}^{(s+1)}], \label{relation 2 in the system I}
\end{align}

\begin{align}
[ \mathcal{S}_{k, \ell-1}^{(s)} ] [ \mathcal{S}_{k-1, \ell}^{(s+2)} ] =[ \mathcal{S}_{k, \ell}^{(s)}] [ \mathcal{S}_{k-1, \ell-1}^{(s+2)}] + [ \mathcal{\tilde{T}}_{k, 0, \lfloor \frac{\ell}{2} \rfloor}^{(s+1)}] [ \mathcal{T}_{\lfloor \frac{k}{2} \rfloor, 0, \ell}^{(s+2\sigma(k)+1)}] [ \mathcal{T}_{\lfloor \frac{\ell-1}{2} \rfloor, 0, 0}^{(s+2k+7)}] [ \mathcal{T}_{\lfloor \frac{k-1}{2} \rfloor, 0, 0}^{(s+2\sigma(k+1)+1)}], \label{relation 3 in the system I}
\end{align}

\begin{align}
[ \mathcal{\tilde{T}}_{k, 0, m-1}^{(s)} ] [ \mathcal{\tilde{T}}_{k-1, 0, m}^{(s+2)} ] =[ \mathcal{\tilde{T}}_{k, 0, m}^{(s)} ] [ \mathcal{\tilde{T}}_{k-1, 0, m-1}^{(s+2)} ] + [ \mathcal{S}_{k-1, 2m-1}^{(s+1)} ], \label{relation 4 in the system I}
\end{align}

\begin{align}
[ \mathcal{T}_{0, \ell, r-1}^{(s)} ] [ \mathcal{T}_{0, \ell-1, r}^{(s+2)} ] =[ \mathcal{T}_{0, \ell, r}^{(s)} ] [ \mathcal{T}_{0, \ell-1, r-1}^{(s+2)} ] + [ \mathcal{T}_{0, 0,  \ell-1}^{(s-1)}] [ \mathcal{T}_{\lfloor \frac{\ell}{2} \rfloor, r-1, 0}^{(s+2\sigma(\ell)+2)} ] [ \mathcal{T}_{\lfloor \frac{\ell+1}{2} \rfloor, 0, 0}^{(s+2\sigma(\ell+1)+2)} ], \label{relation 5 in the system I}
\end{align}

\begin{align}
[ \mathcal{R}_{k, 2\ell, \ell-1}^{(s)} ] [ \mathcal{R}_{k-1, 2\ell, \ell}^{(s+2)} ] =[ \mathcal{R}_{k, 2\ell, \ell}^{(s)} ] [ \mathcal{R}_{k-1, 2\ell, \ell-1}^{(s+2)} ] + [ \mathcal{T}_{0, 0,  k+2\ell}^{(s-1)} ] [ \mathcal{T}_{\lfloor \frac{k-1}{2} \rfloor, 0, 2\ell}^{(s+2\sigma(k+1)+1)} ] [ \mathcal{T}_{\lfloor \frac{k}{2} \rfloor, 2\ell-1, 0}^{(s+2\sigma(k)+1)} ], \label{relation 6 in the system I}
\end{align}

\begin{align}
[ \mathcal{R}_{k, 2\ell, \ell}^{(s)} ] [ \mathcal{R}_{k-1, 2\ell+1, \ell}^{(s+2)} ] =[ \mathcal{R}_{k, 2\ell+1, \ell}^{(s)} ] [ \mathcal{R}_{k-1, 2\ell, \ell}^{(s+2)} ] + [ \mathcal{\tilde{T}}_{k-1, 0,  \ell}^{(s+1)} ] [ \mathcal{T}_{\lfloor \frac{(k+1)}{2} \rfloor + \ell, 0, 0}^{(s+2\sigma(k+1)+1)} ] [ \mathcal{T}_{\lfloor \frac{k}{2} \rfloor, 2\ell, 0}^{(s+2\sigma(k)+1)} ], \label{relation 7 in the system I}
\end{align}

\begin{align}
[ \mathcal{R}_{k, 2\ell+1, \ell-1}^{(s)} ] [ \mathcal{R}_{k-1, 2\ell+2, \ell}^{(s+2)} ] =[ \mathcal{R}_{k, 2\ell+2, \ell}^{(s)} ] [ \mathcal{R}_{k-1, 2\ell+1, \ell}^{(s+2)} ] + [ \mathcal{\tilde{T}}_{k-1, 0,  \ell}^{(s+1)} ] [ \mathcal{T}_{\lfloor \frac{k+1}{2} \rfloor + \ell, 0, 0}^{(s+2\sigma(k+1)+1)} ] [ \mathcal{T}_{\lfloor \frac{k}{2} \rfloor, 2\ell+1, 0}^{(s+2\sigma(k)+1)} ], \label{relation 8 in the system I}
\end{align}

\begin{align}
[ \mathcal{R}_{0, 2\ell+i, \ell}^{(s)} ] = [ \mathcal{T}_{0, 0, 2\ell+i}^{(s-2)} ] [ \mathcal{T}_{\ell, 0, 0}^{(s+2)} ], \ i=0, 1, 2. \label{relation 9 in the system I}
\end{align}
\end{theorem}

We will prove Theorem \ref{the system I} in Section \ref{prove system}.

From the relations (\ref{relation 6 in the system I}), (\ref{relation 7 in the system I}), and (\ref{relation 8 in the system I}), we see that if $\mathcal{R}_{k, \ell, 0}^{(s)}$ is a module in the relations of Theorem \ref{the system I} for some $k, \ell, s$, then $\ell \in \{0, 1, 2\}$. Therefore we need the relations for the modules $\mathcal{R}_{k, \ell, 0}^{(s)}=\mathcal{T}_{ 0, k , \ell}^{(s-1)}$, $\ell \in \{0, 1, 2\}$, in order to obtain a closed system. These relations are (\ref{relation 5 in the system I}).

Let us use I to denote the system consisting of the relations in Theorem \ref{the system I} and the relations (\ref{trivial relations 1}), (\ref{usual T equation 1})-(\ref{usual T equation 3}). Then the system I is closed.

All relations except (\ref{relation 9 in the system I}) in Theorem \ref{the system I} are written in the form $[\mathcal{L}] [\mathcal{R}]=[\mathcal{T}] [\mathcal{B}] + [\mathcal{S}]$, where $\mathcal{L}, \mathcal{R}, \mathcal{T}, \mathcal{B}$ are irreducible modules which are called \textit{left, right, top} and \textit{bottom modules} and $\mathcal{S}$ is a tensor product of some irreducible modules. The factors of $\mathcal{S}$ are called \textit{sources}, see \cite{MY12b}, \cite{LM12}. Moreover, we have the following theorem.

\begin{theorem} \label{irreducible}
For each relation in Theorem \ref{the system I}, all summands on the right hand side, $\mathcal{T}\otimes \mathcal{B}$ and $\mathcal{S}$, are irreducible.
\end{theorem}
We will prove Theorem \ref{irreducible} in Section \ref{prove irreducible}.

The system I can be used to compute the $q$-characters of the modules in the system. We have the following proposition.
\begin{proposition} \label{compute 1}
One can compute the $q$-characters of $\mathcal{T}_{k, \ell, 0}^{(s)}$, $\mathcal{T}_{k, 0, m}^{(s)}$, $\mathcal{T}_{0, \ell, r}^{(s)}$, $\mathcal{\tilde{T}}_{k, 0, m}^{(s)}$, $\mathcal{S}_{k, \ell}^{(s)}$, $\mathcal{R}_{k, 2\ell+j, \ell}^{(s)}$, $j=0, 1, 2$, $s\in\mathbb{Z}$, $k, \ell, m \in  \mathbb{Z}_{\geq 0}$, recursively, from $\chi_q(1_0), \chi_q(2_0)$, and $\chi_q(3_0)$, by using the relations in the system I.
\end{proposition}
We will prove Proposition \ref{compute 1} in Section \ref{prove compute}.

\subsection{The system II}
We will describe the system II which is dual to the system I. Let $\bar{T}_{k, \ell, m}^{(s)}, \bar{\tilde{T}}_{k, \ell, m}^{(s)}, \bar{S}_{k, \ell}^{(s)}, \bar{R}_{k, \ell, m}^{(s)}$ be the monomials obtained from $T_{k, \ell, m}^{(s)}$, $\tilde{T}_{k, \ell, m}^{(s)}$, $S_{k, \ell}^{(s)}$, $R_{k, \ell, m}^{(s)}$ by replacing $i_a$ with $i_{-a}$, $i=1, 2, 3$. Namely,
\begin{align*}
\bar{T}_{k, \ell, m}^{(s)} = \left( \prod_{i=0}^{m-1} 1_{-s-4k-2\ell-2i-2} \right) \left( \prod_{i=0}^{\ell-1} 2_{-s-4k-2i-1} \right) \left( \prod_{i=0}^{k-1} 3_{-s-4i} \right),
\end{align*}

\begin{align*}
\bar{\tilde{T}}_{k, \ell, m}^{(s)} = \left( \prod_{i=0}^{m-1} 3_{-s-2k-2\ell-4i-4} \right) \left( \prod_{i=0}^{\ell-1} 2_{-s-2k-2i-1} \right) \left( \prod_{i=0}^{k-1} 1_{-s-2i} \right),
\end{align*}

\begin{align*}
\bar{S}_{k, \ell}^{(s)} = \left( \prod_{i=0}^{\ell-1} 2_{-s-2k-2i-4} \right) \left( \prod_{i=0}^{k-1} 2_{-s-2i} \right),
\end{align*}

\begin{align*}
\bar{R}_{k, \ell, m}^{(s)} =   \left( \prod_{i=0}^{m-1} 3_{-s-2k-4i-3} \right) \left( \prod_{i=0}^{\ell-1} 1_{-s-2k-2i-1} \right) \left( \prod_{i=0}^{k-1} 2_{-s-2i} \right).
\end{align*}

For $k, \ell \in \mathbb{Z}_{\geq 0}, s \in \mathbb{Z}$, we have the following trivial relations
\begin{equation}
\begin{aligned}
& \mathcal{\bar{\tilde{T}}}_{k, 0, 0}^{(s)} = \mathcal{T}_{0, 0, k}^{(-s-2k)}, \ \mathcal{\bar{\tilde{T}}}_{0, k, 0}^{(s)} = \mathcal{T}_{0, k, 0}^{(-s-2k)}, \ \mathcal{\bar{\tilde{T}}}_{0, 0, k}^{(s)} = \mathcal{T}_{k, 0, 0}^{(-s-4k)}, \\
& \mathcal{\bar{S}}_{0, k}^{(s)}=\mathcal{\bar{S}}_{k, 0}^{(s+4)}=\mathcal{T}_{0, k, 0}^{(-s-2k-3)}, \\
& \mathcal{\bar{R}}_{ k , \ell, 0}^{(s)} = \mathcal{\bar{T}}_{0,  k , \ell}^{(s-1)}, \ \mathcal{\bar{R}}_{k, 0, 0}^{(s)}=\mathcal{T}_{0, k, 0}^{(-s-2k+1)}, \ \mathcal{\bar{R}}_{0, k, 0}^{(s)}=\mathcal{T}_{0, 0, k}^{(-s-2k-1)}, \ \mathcal{\bar{R}}_{0, 0, k}^{(s)}=\mathcal{T}_{k, 0, 0}^{(-s-4k+1)}. \label{trivial relations 2}
\end{aligned}
\end{equation}

\begin{theorem} \label{anti-special II}
The modules $\mathcal{\bar{T}}_{k, \ell, 0}^{(s)}, \mathcal{\bar{T}}_{k, 0, m}^{(s)}, \mathcal{\bar{T}}_{0, \ell, r}^{(s)}, \mathcal{\bar{\tilde{T}}}_{k, 0, m}^{(s)}, \mathcal{\bar{S}}_{k, \ell}^{(s)}, \mathcal{\bar{R}}_{k, 2\ell+j, \ell}^{(s)}$, $k, l, m\in\mathbb{Z}_{\geq 0}, r \in \{1, 2\}, j \in \{0, 1, 2\}$, are anti-special.
\end{theorem}

\begin{proof}
This theorem can be proved using the dual arguments of the proof of Theorem \ref{special}.
\end{proof}

\begin{lemma} \label{involution II}
Let $\iota: \mathbb{Z}\mathcal{P} \to \mathbb{Z}\mathcal{P}$ be a ring homomorphism such that $Y_{i, aq^{s}} \mapsto Y_{i, aq^{8-s}}^{-1}$, $i=1, 2, 3$, for all $a \in \mathbb{C}^{\times}, s\in \mathbb{Z}$. Let $m_+$ be one of the monomials
\begin{align*}
T_{k, \ell, 0}^{(s)}, \ T_{k, 0, m}^{(s)}, \ T_{0, \ell, r}^{(s)}, \ \tilde{T}_{k, 0, m}^{(s)}, \ S_{k, \ell}^{(s)}, \ R_{k, 2\ell+j, \ell}^{(s)}, \ j=0, 1, 2.
\end{align*}
Then $ \chi_q(\bar{m_+}) = \iota(\chi_q(m_+)) $.
\end{lemma}

\begin{proof}
The lemma is proved by using similar arguments of the proof of Lemma 7.3 of \cite{LM12}.
\end{proof}

The modules $\mathcal{\bar{T}}_{k, \ell, 0}^{(s)}, \mathcal{\bar{T}}_{k, 0, m}^{(s)}, \mathcal{\bar{T}}_{0, \ell, r}^{(s)}, \mathcal{\bar{\tilde{T}}}_{k, 0, m}^{(s)}, \mathcal{\bar{S}}_{k, \ell}^{(s)}, \mathcal{\bar{R}}_{k, 2\ell+j, \ell}^{(s)}$ satisfy the same relations as in Theorem \ref{the system I} but the roles of left and right modules are exchanged.
\begin{theorem} \label{the system II}
For $k, \ell, m \in \mathbb{Z}_{\geq 1}$, $r \in \{1, 2\}$, we have the following relations in $\rep(U_q \hat{\mathfrak{g}})$.
\begin{align}
 [ \mathcal{\bar{T}}_{k-1, \ell, 0}^{(s+4)} ] [ \mathcal{\bar{T}}_{k, \ell-1, 0}^{(s)} ] =[ \mathcal{\bar{T}}_{k, \ell, 0}^{(s)} ] [ \mathcal{\bar{T}}_{k-1, \ell-1, 0}^{(s+4)} ] + [ \mathcal{\bar{R}}_{2k-1, \ell,  \lfloor \frac{\ell}{2} \rfloor}^{(s+1)} ] [ \mathcal{\bar{T}}_{\lfloor \frac{\ell-1}{2} \rfloor, 0, 0}^{(s+4k+4)} ], \label{relation 1 in the system II}
\end{align}

\begin{align}
[ \mathcal{\bar{T}}_{k-1, 0, m}^{(s+4)} ] [ \mathcal{\bar{T}}_{k, 0, m-1}^{(s)} ] =[ \mathcal{\bar{T}}_{k, 0, m}^{(s)} ] [ \mathcal{\bar{T}}_{k-1, 0, m-1}^{(s+4)} ] + [ \mathcal{\bar{S}}_{2k-1, m-1}^{(s+1)}], \label{relation 2 in the system II}
\end{align}

\begin{align}
 [ \mathcal{\bar{S}}_{k-1, \ell}^{(s+2)} ] [ \mathcal{\bar{S}}_{k, \ell-1}^{(s)} ] =[ \mathcal{\bar{S}}_{k, \ell}^{(s)}] [ \mathcal{\bar{S}}_{k-1, \ell-1}^{(s+2)}] + [ \mathcal{\bar{\tilde{T}}}_{k, 0, \lfloor \frac{\ell}{2} \rfloor}^{(s+1)}] [ \mathcal{\bar{T}}_{\lfloor \frac{k}{2} \rfloor, 0, \ell}^{(s+2\sigma(k)+1)}] [ \mathcal{\bar{T}}_{\lfloor \frac{\ell-1}{2} \rfloor, 0, 0}^{(s+2k+7)}] [ \mathcal{\bar{T}}_{\lfloor \frac{k-1}{2} \rfloor, 0, 0}^{(s+2\sigma(k+1)+1)}],  \label{relation 3 in the system II}
\end{align}

\begin{align}
 [ \mathcal{\bar{\tilde{T}}}_{k-1, 0, m}^{(s+2)} ] [ \mathcal{\bar{\tilde{T}}}_{k, 0, m-1}^{(s)} ] =[ \mathcal{\bar{\tilde{T}}}_{k, 0, m}^{(s)} ] [ \mathcal{\bar{\tilde{T}}}_{k-1, 0, m-1}^{(s+2)} ] + [ \mathcal{\bar{S}}_{k-1, 2m-1}^{(s+1)} ],                    \label{relation 4 in the system II}
\end{align}

\begin{align}
 [ \mathcal{\bar{T}}_{0, \ell-1, r}^{(s+2)} ] [ \mathcal{\bar{T}}_{0, \ell, r-1}^{(s)} ] =[ \mathcal{\bar{T}}_{0, \ell, r}^{(s)} ] [ \mathcal{\bar{T}}_{0, \ell-1, r-1}^{(s+2)} ] + [ \mathcal{\bar{T}}_{0, 0,  \ell-1}^{(s-1)}] [ \mathcal{\bar{T}}_{\lfloor \frac{\ell}{2} \rfloor, r-1, 0}^{(s+2\sigma(\ell)+2)} ] [ \mathcal{\bar{T}}_{\lfloor \frac{\ell+1}{2} \rfloor, 0, 0}^{(s+2\sigma(\ell+1)+2)} ],      \label{relation 5 in the system II}
\end{align}

\begin{align}
[ \mathcal{\bar{R}}_{k-1, 2\ell, \ell}^{(s+2)} ] [ \mathcal{\bar{R}}_{k, 2\ell, \ell-1}^{(s)} ] =[ \mathcal{\bar{R}}_{k, 2\ell, \ell}^{(s)} ] [ \mathcal{\bar{R}}_{k-1, 2\ell, \ell-1}^{(s+2)} ] + [ \mathcal{\bar{T}}_{0, 0,  k+2\ell}^{(s-1)} ] [ \mathcal{\bar{T}}_{\lfloor \frac{k-1}{2} \rfloor, 0, 2\ell}^{(s+2\sigma(k+1)+1)} ] [ \mathcal{\bar{T}}_{\lfloor \frac{k}{2} \rfloor, 2\ell-1, 0}^{(s+2\sigma(k)+1)} ],       \label{relation 6 in the system II}
\end{align}

\begin{align}
 [ \mathcal{\bar{R}}_{k-1, 2\ell+1, \ell}^{(s+2)} ] [ \mathcal{\bar{R}}_{k, 2\ell, \ell}^{(s)} ] =[ \mathcal{\bar{R}}_{k, 2\ell+1, \ell}^{(s)} ] [ \mathcal{\bar{R}}_{k-1, 2\ell, \ell}^{(s+2)} ] + [ \mathcal{\bar{\tilde{T}}}_{k-1, 0,  \ell}^{(s+1)} ] [ \mathcal{\bar{T}}_{\lfloor \frac{(k+1)}{2} \rfloor + \ell, 0, 0}^{(s+2\sigma(k+1)+1)} ] [ \mathcal{\bar{T}}_{\lfloor \frac{k}{2} \rfloor, 2\ell, 0}^{(s+2\sigma(k)+1)} ],  \label{relation 7 in the system II}
\end{align}

\begin{align}
[ \mathcal{\bar{R}}_{k-1, 2\ell+2, \ell}^{(s+2)} ] [ \mathcal{\bar{R}}_{k, 2\ell+1, \ell-1}^{(s)} ]  =[ \mathcal{\bar{R}}_{k, 2\ell+2, \ell}^{(s)} ] [ \mathcal{\bar{R}}_{k-1, 2\ell+1, \ell}^{(s+2)} ] + [ \mathcal{\bar{\tilde{T}}}_{k-1, 0,  \ell}^{(s+1)} ] [ \mathcal{\bar{T}}_{\lfloor \frac{k+1}{2} \rfloor + \ell, 0, 0}^{(s+2\sigma(k+1)+1)} ] [ \mathcal{\bar{T}}_{\lfloor \frac{k}{2} \rfloor, 2\ell+1, 0}^{(s+2\sigma(k)+1)} ], \label{relation 8 in the system II}
\end{align}

\begin{align}
[ \mathcal{\bar{R}}_{0, 2\ell+i, \ell}^{(s)} ] = [ \mathcal{\bar{T}}_{0, 0, 2\ell+i}^{(s-2)} ] [ \mathcal{\bar{T}}_{\ell, 0, 0}^{(s+2)} ], \ i=0, 1, 2.   \label{relation 9 in the system II}
\end{align}
Moreover, the modules corresponding to each summand on the right hand side of the above relations are all irreducible.
\end{theorem}

\begin{proof}
The theorem follows from the relations in Theorem \ref{the system I}, Theorem \ref{irreducible}, and Lemma \ref{involution II}.
\end{proof}

Let us use II to denote the system consisting of the relations in Theorem \ref{the system II} and the relations (\ref{trivial relations 2}), (\ref{usual T equation 1})-(\ref{usual T equation 3}). Then the system II is closed.

Since $\mathcal{\bar{T}}_{k, \ell, 0}^{(s)} = \mathcal{\tilde{T}}_{0, \ell, k}^{(-s-4k-2\ell)}$, the system II contains minimal affinizations $\mathcal{\tilde{T}}_{0, \ell, k}^{(-s-4k-2\ell)}$, $k, \ell \in \mathbb{Z}_{\geq 0}$, $s\in \mathbb{Z}$. Using the shift defined in (\ref{shift}), we can obtain minimal affinizations $\mathcal{\tilde{T}}_{0, \ell, k}^{(s)}$ from $\mathcal{\tilde{T}}_{0, \ell, k}^{(-s-4k-2\ell)}$, $k, \ell \in \mathbb{Z}_{\geq 0}$, $s\in \mathbb{Z}$.

The following proposition is similar to Proposition \ref{compute 1}.
\begin{proposition} \label{compute II}
One can compute the $q$-characters of
\begin{align*}
\mathcal{\bar{T}}_{k, \ell, 0}^{(s)}, \mathcal{\bar{T}}_{k, 0, m}^{(s)}, \mathcal{\bar{T}}_{0, \ell, r}^{(s)}, \mathcal{\bar{\tilde{T}}}_{k, 0, m}^{(s)}, \mathcal{\bar{S}}_{k, \ell}^{(s)}, \mathcal{\bar{R}}_{k, 2\ell+j, \ell}^{(s)}, \ s\in\mathbb{Z}, \ k, \ell, m \in  \mathbb{Z}_{\geq 0}, r\in \{1, 2\},
\end{align*}
recursively, from $\chi_q(1_0), \chi_q(2_0), \chi_q(3_0)$ by using the relations in the system II. $\Box$
\end{proposition}

\subsection{The system III}
For $s\in \mathbb{Z}, k, \ell \in \mathbb{Z}_{\geq 0}$, we define the following monomials.
\begin{align*}
U_{k, \ell}^{(s)} = \left( \prod_{i=0}^{k-1} 2_{s+2i} \right) \left( \prod_{i=0}^{\ell-1} 3_{s+2k+2i+1} \right),
\end{align*}

\begin{align*}
V_{k, \ell}^{(s)} = \left( \prod_{i=0}^{k-1} 3_{s+4i} \right) \left( \prod_{i=0}^{\ell-1} 3_{s+4k+4i+2} \right),
\end{align*}

\begin{align*}
P_{k, \ell}^{(s)} = \left( \prod_{i=0}^{k-1} 3_{s+4i} \right) 2_{4k+1} \left( \prod_{i=0}^{\ell-1} 3_{s+4k+4i+6} \right),
\end{align*}

\begin{align*}
O_{k, \ell}^{(s)} = \left( \prod_{i=0}^{k-1} 2_{s+2i} \right) 1_{2k+1} 1_{2k+3} \left( \prod_{i=0}^{\ell-1} 2_{s+2k+2i+6} \right).
\end{align*}

For $k \in \mathbb{Z}_{\geq 0}, s \in \mathbb{Z}$, we have the following trivial relations
\begin{equation}
\begin{aligned}
& \mathcal{U}_{k, 0}^{(s)} = \mathcal{T}_{0, k, 0}^{(s-1)}, \ \mathcal{V}_{k, 0}^{(s)} = \mathcal{T}_{k, 0, 0}^{(s)}, \ \mathcal{V}_{0, k}^{(s)} = \mathcal{T}_{k, 0, 0}^{(s+2)}, \\
& \mathcal{O}_{k, 0}^{(s)} = \mathcal{T}_{0, k, 2}^{(s-1)}, \ \mathcal{O}_{0, k}^{(s)} = \mathcal{\tilde{T}}_{2, k, 0}^{(s+6)}, \ \mathcal{P}_{k, 0}^{(s)} = \mathcal{T}_{k, 1, 0}^{(s)}, \ \mathcal{P}_{0, k}^{(s)} = \mathcal{\tilde{T}}_{0, 1, k}^{(s)}. \label{trivial relations 3}
\end{aligned}
\end{equation}

\begin{theorem}[{\cite{Her07}}] \label{special already known}
The modules $\mathcal{\tilde{T}}_{k, \ell, 0}^{(s)}$, $s\in \mathbb{Z}$, $k, \ell \in \mathbb{Z}_{\geq 0}$, are special.
\end{theorem}

\begin{theorem} \label{special in the system III}
The modules $\mathcal{U}_{k, \ell}^{(s)}$, $\mathcal{V}_{k, \ell}^{(s)}$, $\mathcal{P}_{k, \ell}^{(s)}$, $\mathcal{O}_{k, \ell}^{(s)}$, $s\in \mathbb{Z}$, $k, \ell \in \mathbb{Z}_{\geq 0}$, are special.
\end{theorem}
We will prove Theorem \ref{special in the system III} in Section \ref{prove special}.

Our second main result is the following theorem.
\begin{theorem} \label{the system III}
For $s\in \mathbb{Z}$, $k, \ell \in \mathbb{Z}_{\geq 1}$, $p \in \mathbb{Z}_{\geq 2}$, we have the following relations in $\rep(U_q \hat{\mathfrak{g}})$.
\begin{align}
[ \mathcal{\tilde{T}}_{k, \ell-1, 0}^{(s)} ] [ \mathcal{\tilde{T}}_{k-1, \ell, 0}^{(s+2)} ] =[ \mathcal{\tilde{T}}_{k, \ell, 0}^{(s)} ] [ \mathcal{\tilde{T}}_{k-1, \ell-1, 0}^{(s+2)} ] + [ \mathcal{\tilde{T}}_{\ell-1, 0, 0}^{(s+2k+2)} ] [ \mathcal{U}_{k-1, \ell}^{(s+1)} ], \label{relation 1 in system 3}
\end{align}

\begin{align}
[ \mathcal{U}_{r, \ell}^{(s)} ] = [ \mathcal{ \tilde{T}}_{0, r, \lfloor \frac{\ell}{2} \rfloor }^{(s-1)} ] [ \mathcal{T}_{ \lfloor \frac{\ell+1}{2} \rfloor, 0, 0 }^{(s+2r+1)} ], \quad r=0, 1, \label{relation 2 in system 3}
\end{align}

\begin{align}
[ \mathcal{U}_{p, \ell-1}^{(s)} ] [ \mathcal{U}_{p-1, \ell}^{(s+2)} ] =[ \mathcal{U}_{p, \ell}^{(s)} ] [ \mathcal{U}_{p-1, \ell-1}^{(s+2)} ] +
\begin{cases}
[ \mathcal{\tilde{T}}_{p, \ell-1, 0}^{(s+1)} ] [ \mathcal{T}_{\lfloor \frac{p}{2} \rfloor + \frac{\ell}{2}, 0, 0}^{(s+2\sigma(p)+1)} ] [ \mathcal{V}_{\lfloor \frac{p-1}{2} \rfloor, \frac{\ell}{2}}^{(s+2\sigma(p+1)+1)}], & \text{if } \ell \text{ is even}, \\
[ \mathcal{\tilde{T}}_{p, \ell-1, 0}^{(s+1)} ] [ \mathcal{T}_{\lfloor \frac{p+1}{2} \rfloor + \frac{\ell-1}{2}, 0, 0}^{(s+2\sigma(p+1)+1)} ] [ \mathcal{P}_{\lfloor \frac{p-2}{2} \rfloor, \frac{\ell-1}{2}}^{(s+2\sigma(p)+1)}], & \text{if } \ell \text{ is odd},
\end{cases}                    \label{relation 3 in system 3}
\end{align}

\begin{align}
[ \mathcal{V}_{k, \ell-1}^{(s)} ] [ \mathcal{V}_{k-1, \ell}^{(s+4)} ] =[ \mathcal{V}_{k, \ell}^{(s)} ] [ \mathcal{V}_{k-1, \ell-1}^{(s+4)} ] +
[ \mathcal{O}_{2(k-1), 2(\ell-1)}^{(s+1)} ],         \label{relation 4 in system 3}
\end{align}

\begin{align}
[ \mathcal{O}_{k, \ell-1}^{(s)} ] [ \mathcal{O}_{k-1, \ell}^{(s+2)} ] =[ \mathcal{O}_{k, \ell}^{(s)} ] [ \mathcal{O}_{k-1, \ell-1}^{(s+2)} ] +
[ \mathcal{P}_{\lfloor \frac{k}{2} \rfloor, \lfloor \frac{\ell}{2} \rfloor}^{(s+\sigma(k)+1)} ] [ \mathcal{V}_{\lfloor \frac{k+1}{2} \rfloor, \lfloor \frac{\ell+1}{2} \rfloor}^{(s+\sigma(k+1)+1)} ] [ \mathcal{\tilde{T}}_{k-1, 0, 0}^{(s+1)} ] [ \mathcal{\tilde{T}}_{\ell-1, 0, 0}^{(s+2k+7)} ],          \label{relation 5 in system 3}
\end{align}

\begin{align}
[ \mathcal{P}_{k, \ell-1}^{(s)} ] [ \mathcal{P}_{k-1, \ell}^{(s+2)} ] =[ \mathcal{P}_{k, \ell}^{(s)} ] [ \mathcal{P}_{k-1, \ell-1}^{(s+2)} ] +
[ \mathcal{O}_{2k-1, 2\ell-1}^{(s+1)} ].              \label{relation 6 in system 3}
\end{align}
\end{theorem}

We will prove Theorem \ref{the system III} in Section \ref{prove system}.

Note that the relations for the modules $\mathcal{O}_{0, k}^{(s)} = \mathcal{\tilde{T}}_{2, k, 0}^{(s+6)}$ are contained in (\ref{relation 1 in system 3}). The modules $\mathcal{T}_{0, k, 2}^{(s-1)}$, $\mathcal{T}_{k, 1, 0}^{(s)}$ in (\ref{trivial relations 3}) can be computed by using the relations (\ref{relation 1 in the system I}, $\ell=1$), (\ref{relation 5 in the system I}), (\ref{relation 7 in the system I}, $\ell=0$) in the system I. The modules $\mathcal{\tilde{T}}_{0, 1, k}^{(s)}$ in (\ref{trivial relations 3}) can be computed by using the relations (\ref{relation 1 in the system II}, $\ell=1$) and (\ref{relation 7 in the system II}, $\ell=0$) in the system II.

Let us use III to denote the system consisting of the relations in Theorem \ref{the system III},  the relations (\ref{relation 1 in the system I}, $\ell=1$), (\ref{relation 5 in the system I}), (\ref{relation 7 in the system I}, $\ell=0$) in the system I, the relations (\ref{relation 1 in the system II}, $\ell=1$), (\ref{relation 7 in the system II}, $\ell=0$) in the system II, and the relations (\ref{trivial relations 3}), (\ref{usual T equation 1})-(\ref{usual T equation 3}). Then the system III is closed.

All relations except (\ref{relation 2 in system 3}) in Theorem \ref{the system III} are written in the form $[\mathcal{L}] [\mathcal{R}]=[\mathcal{T}] [\mathcal{B}] + [\mathcal{S}]$. We have the following theorem.
\begin{theorem} \label{irreducible 3}
For each relation in Theorem \ref{the system III}, all summands on the right hand side, $\mathcal{T}\otimes \mathcal{B}$ and $\mathcal{S}$, are irreducible.
\end{theorem}
We will prove Theorem \ref{irreducible 3} in Section \ref{prove irreducible}.

The system III can be used to compute the $q$-characters of the modules in the system III. The following proposition can be proved by similar arguments as the proof of Proposition \ref{compute 1}.
\begin{proposition} \label{compute 3}
One can compute the $q$-characters of $\mathcal{U}_{k, \ell}^{(s)}$, $\mathcal{V}_{k, \ell}^{(s)}$, $\mathcal{P}_{k, \ell}^{(s)}$, $\mathcal{O}_{k, \ell}^{(s)}$, $s\in \mathbb{Z}$, $k, \ell \in \mathbb{Z}_{\geq 0}$, recursively, from $\chi_q(1_0), \chi_q(2_0)$, and $\chi_q(3_0)$ by using the relations in the system III. $\Box$
\end{proposition}

\subsection{The system IV}
We will describe the system IV which is dual to the system III.
Let $\bar{U}_{k, \ell}^{(s)}, \bar{V}_{k, \ell}^{(s)}, \bar{P}_{k, \ell}^{(s)}, \bar{O}_{k, \ell}^{(s)}$ be the monomials obtained from $U_{k, \ell}^{(s)}$, $V_{k, \ell}^{(s)}$, $P_{k, \ell}^{(s)}$, $O_{k, \ell}^{(s)}$ by replacing $i_a$ with $i_{-a}$, $i=1, 2, 3$. Namely,
\begin{align*}
\bar{U}_{k, \ell}^{(s)} = \left( \prod_{i=0}^{\ell-1} 3_{-s-2k-2i-1} \right) \left( \prod_{i=0}^{k-1} 2_{-s-2i} \right),
\end{align*}

\begin{align*}
\bar{V}_{k, \ell}^{(s)} = \left( \prod_{i=0}^{\ell-1} 3_{-s-4k-4i-2} \right) \left( \prod_{i=0}^{k-1} 3_{-s-4i} \right),
\end{align*}

\begin{align*}
\bar{P}_{k, \ell}^{(s)} = \left( \prod_{i=0}^{\ell-1} 3_{-s-4k-4i-6} \right)  2_{-4k-1} \left( \prod_{i=0}^{k-1} 3_{-s-4i} \right),
\end{align*}

\begin{align*}
\bar{O}_{k, \ell, m}^{(s)} = \left( \prod_{i=0}^{\ell-1} 2_{-s-2k-2i-6} \right)  1_{-2k-1} 1_{-2k-3} \left( \prod_{i=0}^{k-1} 2_{-s-2i} \right).
\end{align*}

For $k \in \mathbb{Z}_{\geq 0}, s \in \mathbb{Z}$, we have the following trivial relations
\begin{equation}
\begin{aligned}
& \mathcal{\bar{U}}_{k, 0}^{(s)} = \mathcal{\bar{T}}_{0, k, 0}^{(s-1)}, \ \mathcal{\bar{V}}_{k, 0}^{(s)} = \mathcal{\bar{T}}_{k, 0, 0}^{(s)}, \ \mathcal{\bar{V}}_{0, k}^{(s)} = \mathcal{\bar{T}}_{k, 0, 0}^{(s+2)}, \\
& \mathcal{\bar{O}}_{k, 0}^{(s)} = \mathcal{\bar{T}}_{0, k, 2}^{(s-1)}, \ \mathcal{\bar{O}}_{0, k}^{(s)} = \mathcal{\bar{\tilde{T}}}_{2, k, 0}^{(s+6)}, \\
& \mathcal{\bar{P}}_{k, 0}^{(s)} = \mathcal{\bar{T}}_{k, 1, 0}^{(s)}, \ \mathcal{\bar{P}}_{0, k}^{(s)} = \mathcal{\bar{\tilde{T}}}_{0, 1, k}^{(s)} = \mathcal{T}_{k, 1, 0}^{(-s-4k-2)}. \label{trivial relations 4}
\end{aligned}
\end{equation}

The following theorem is proved by using the dual arguments of the proof of Theorem \ref{special in the system III}.
\begin{theorem} \label{anti-special in the system IV}
The modules $\mathcal{\bar{\tilde{T}}}_{k, \ell, 0}^{(s)}$, $\mathcal{\bar{U}}_{k, \ell}^{(s)}$, $\mathcal{\bar{V}}_{k, \ell}^{(s)}$, $\mathcal{\bar{P}}_{k, \ell}^{(s)}$, $\mathcal{\bar{O}}_{k, \ell}^{(s)}$, $s\in \mathbb{Z}$, $k, \ell \in \mathbb{Z}_{\geq 0}$, are anti-special. $\Box$
\end{theorem}

\begin{lemma} \label{involution IV}
Let $\iota: \mathbb{Z}\mathcal{P} \to \mathbb{Z}\mathcal{P}$ be a homomorphism of rings such that $Y_{i, aq^{s}} \mapsto Y_{i, aq^{8-s}}^{-1}$, $i=1, 2, 3$, for all $a \in \mathbb{C}^{\times}, s\in \mathbb{Z}$. Let $m_+$ be one of the monomials
\begin{align*}
\tilde{T}_{k, \ell, 0}^{(s)}, \ U_{k, \ell}^{(s)}, \ V_{k, \ell}^{(s)}, \ P_{k, \ell}^{(s)}, \ O_{k, \ell}^{(s)}, \ s\in \mathbb{Z}, k, \ell \in \mathbb{Z}_{\geq 0},
\end{align*}
Then $\chi_q(\bar{m_+}) = \iota(\chi_q(m_+)) $.
\end{lemma}

\begin{proof}
The lemma is proved by using similar arguments of the proof of Lemma 7.3 of \cite{LM12}.
\end{proof}

The modules $\mathcal{\bar{\tilde{T}}}_{k, \ell, 0}^{(s)}, \mathcal{\bar{U}}_{k, \ell}^{(s)}, \mathcal{\bar{V}}_{k, \ell}^{(s)}, \mathcal{\bar{P}}_{k, \ell}^{(s)}, \mathcal{\bar{O}}_{k, \ell}^{(s)}$ satisfy the same relations as in Theorem \ref{the system III} but the roles of left and right modules are exchanged. More precisely, we have the following theorem.
\begin{theorem} \label{the system IV}
For $s\in \mathbb{Z}$, $k, \ell \in \mathbb{Z}_{\geq 1}$, $p \in \mathbb{Z}_{\geq 2}$, we have the following relations in $\rep(U_q \hat{\mathfrak{g}})$.
\begin{align}
 [ \mathcal{\bar{\tilde{T}}}_{k-1, \ell, 0}^{(s+2)} ] [ \mathcal{\bar{\tilde{T}}}_{k, \ell-1, 0}^{(s)} ] =[ \mathcal{\bar{\tilde{T}}}_{k, \ell, 0}^{(s)} ] [ \mathcal{\bar{\tilde{T}}}_{k-1, \ell-1, 0}^{(s+2)} ] + [ \mathcal{\bar{\tilde{T}}}_{\ell-1, 0, 0}^{(s+2k+2)} ] [ \mathcal{\bar{U}}_{k-1, \ell}^{(s+1)} ],  \label{relation 1 in system 4}
\end{align}

\begin{align}
[ \mathcal{\bar{U}}_{r, \ell}^{(s)} ] = [ \mathcal{ \bar{\tilde{T}} }_{0, r, \lfloor \frac{\ell}{2} \rfloor }^{(s-1)} ] [ \mathcal{\bar{T}}_{ \lfloor \frac{\ell+1}{2} \rfloor, 0, 0 }^{(s+2r+1)} ], \quad r=0, 1, \label{relation 2 in system 4}
\end{align}

\begin{align}
[ \mathcal{\bar{U}}_{p-1, \ell}^{(s+2)} ] [ \mathcal{\bar{U}}_{p, \ell-1}^{(s)} ]  =[ \mathcal{\bar{U}}_{p, \ell}^{(s)} ] [ \mathcal{\bar{U}}_{p-1, \ell-1}^{(s+2)} ] +
\begin{cases}
[ \mathcal{\bar{\tilde{T}}}_{p, \ell-1, 0}^{(s+1)} ] [ \mathcal{\bar{T}}_{\lfloor \frac{p}{2} \rfloor + \frac{\ell}{2}, 0, 0}^{(s+2\sigma(p)+1)} ] [ \mathcal{\bar{V}}_{\lfloor \frac{p-1}{2} \rfloor, \frac{\ell}{2}}^{(s+2\sigma(p+1)+1)}], & \text{if } \ell \text{ is even}, \\
[ \mathcal{\bar{\tilde{T}}}_{p, \ell-1, 0}^{(s+1)} ] [ \mathcal{\bar{T}}_{\lfloor \frac{p+1}{2} \rfloor + \frac{\ell-1}{2}, 0, 0}^{(s+2\sigma(p+1)+1)} ] [ \mathcal{\bar{P}}_{\lfloor \frac{p-2}{2} \rfloor, \frac{\ell-1}{2}}^{(s+2\sigma(p)+1)}], & \text{if } \ell \text{ is odd},
\end{cases}    \label{relation 3 in system 4}
\end{align}

\begin{align}
 [ \mathcal{\bar{V}}_{k-1, \ell}^{(s+4)} ] [ \mathcal{\bar{V}}_{k, \ell-1}^{(s)} ] =[ \mathcal{\bar{V}}_{k, \ell}^{(s)} ] [ \mathcal{\bar{V}}_{k-1, \ell-1}^{(s+4)} ] +
[ \mathcal{\bar{O}}_{2(k-1), 2(\ell-1)}^{(s+1)} ],       \label{relation 4 in system 4}
\end{align}

\begin{align}
 [ \mathcal{\bar{O}}_{k-1, \ell}^{(s+2)} ] [ \mathcal{\bar{O}}_{k, \ell-1}^{(s)} ] =[ \mathcal{\bar{O}}_{k, \ell}^{(s)} ] [ \mathcal{\bar{O}}_{k-1, \ell-1}^{(s+2)} ] +
[ \mathcal{\bar{P}}_{\lfloor \frac{k}{2} \rfloor, \lfloor \frac{\ell}{2} \rfloor}^{(s+\sigma(k)+1)} ] [ \mathcal{\bar{V}}_{\lfloor \frac{k+1}{2} \rfloor, \lfloor \frac{\ell+1}{2} \rfloor}^{(s+\sigma(k+1)+1)} ] [ \mathcal{\bar{\tilde{T}}}_{k-1, 0, 0}^{(s+1)} ] [ \mathcal{\bar{\tilde{T}}}_{\ell-1, 0, 0}^{(s+2k+7)} ],        \label{relation 5 in system 4}
\end{align}

\begin{align}
 [ \mathcal{\bar{P}}_{k-1, \ell}^{(s+2)} ] [ \mathcal{\bar{P}}_{k, \ell-1}^{(s)} ] =[ \mathcal{\bar{P}}_{k, \ell}^{(s)} ] [ \mathcal{\bar{P}}_{k-1, \ell-1}^{(s+2)} ] +
[ \mathcal{\bar{O}}_{2k-1, 2\ell-1}^{(s+1)} ].       \label{relation 6 in system 4}
\end{align}
Moreover, the modules corresponding to each summand on the right hand side of the above relations are all irreducible.
\end{theorem}

\begin{proof}
The theorem follows from the relations in Theorem \ref{the system III}, Theorem \ref{irreducible 3}, and Lemma \ref{involution IV}.
\end{proof}

Note that the relations for the modules $\mathcal{\bar{O}}_{0, k}^{(s)} = \mathcal{\bar{\tilde{T}}}_{2, k, 0}^{(s+6)}$ are contained in (\ref{relation 1 in system 4}). The modules $\mathcal{\bar{T}}_{0, k, 2}^{(s-1)}$,  $\mathcal{\bar{T}}_{k, 1, 0}^{(s)}$ in (\ref{trivial relations 4}) can be computed by using the relations (\ref{relation 1 in the system II}, $\ell=1$), (\ref{relation 5 in the system II}), (\ref{relation 7 in the system II}, $\ell=0$) in the system II. The modules $\mathcal{T}_{k, 1, 0}^{(-s-4k-2)}$ in (\ref{trivial relations 4}) can be computed by using the relations (\ref{relation 1 in the system I}, $\ell=1$) and (\ref{relation 7 in the system I}, $\ell=0$) in the system I.

Let us use IV to denote the system consisting of the relations in Theorem \ref{the system IV}, the relations (\ref{relation 1 in the system II}, $\ell=1$), (\ref{relation 5 in the system II}), (\ref{relation 7 in the system II}, $\ell=0$) in the system II, the relations (\ref{relation 1 in the system I}, $\ell=1$), (\ref{relation 7 in the system I}, $\ell=0$) in the system I, and the relations (\ref{trivial relations 3}), (\ref{usual T equation 1})-(\ref{usual T equation 3}). Then the system IV is closed.

Since $\mathcal{\bar{\tilde{T}}}_{k, \ell, 0}^{(s)} = \mathcal{T}_{0, \ell, k}^{(-s-2k-2\ell)}$, the system II contains minimal affinizations $\mathcal{T}_{0, \ell, k}^{(-s-2k-2\ell)}$, $k, \ell \in \mathbb{Z}_{\geq 0}$, $s\in \mathbb{Z}$. Using the shift defined in (\ref{shift}), we can obtain minimal affinizations $\mathcal{T}_{0, \ell, k}^{(s)}$ from $\mathcal{T}_{0, \ell, k}^{(-s-2k-2\ell)}$, $k, \ell \in \mathbb{Z}_{\geq 0}$, $s\in \mathbb{Z}$.

The following proposition is similar to Proposition \ref{compute 1}.
\begin{proposition} \label{compute IV}
One can compute the $q$-characters of
\begin{align*}
\mathcal{\bar{\tilde{T}}}_{k, \ell, 0}^{(s)}, \ \mathcal{\bar{U}}_{k, \ell}^{(s)}, \ \mathcal{\bar{V}}_{k, \ell}^{(s)}, \ \mathcal{\bar{P}}_{k, \ell}^{(s)}, \ \mathcal{\bar{O}}_{k, \ell}^{(s)}, \ s\in\mathbb{Z}, \ k, \ell \in  \mathbb{Z}_{\geq 0}, r\in \{1, 2\},
\end{align*}
recursively, from $\chi_q(1_0), \chi_q(2_0), \chi_q(3_0)$ by using the relations in the system IV. $\Box$
\end{proposition}

\section{Special modules} \label{prove special}
In this section, we prove Theorem \ref{special} and Theorem \ref{special in the system III}. Namely, we will prove that for $s\in \mathbb{Z}$, $k, \ell, m \in \mathbb{Z}_{\geq 0}$, $r \in \{0, 1, 2\}$, the modules
\begin{equation}
\begin{aligned}
& \mathcal{T}_{k, \ell, 0}^{(s)}, \ \mathcal{T}_{k, 0, m}^{(s)}, \ \mathcal{\tilde{T}}_{k, 0, m}^{(s)}, \ \mathcal{T}_{0, \ell, r}^{(s)}, \ \mathcal{S}_{k, \ell}^{(s)}, \ \mathcal{R}_{k, 2\ell, \ell}^{(s)}, \\
& \mathcal{R}_{k, 2\ell+1, \ell}^{(s)}, \ \mathcal{R}_{k, 2\ell+2, \ell}^{(s)}, \ \mathcal{U}_{k, \ell}^{(s)}, \ \mathcal{V}_{k, \ell}^{(s)}, \ \mathcal{P}_{k, \ell}^{(s)}, \ \mathcal{O}_{k, \ell}^{(s)},
\end{aligned}   \label{need to prove to be special}
\end{equation}
are special. Since the modules
\begin{align*}
& \mathcal{T}_{k, 0, 0}^{(s)}, \ \mathcal{T}_{0, k, 0}^{(s)}, \ \mathcal{T}_{0, 0, k}^{(s)}, \ \mathcal{\tilde{T}}_{k, 0, 0}^{(s)}, \ \mathcal{\tilde{T}}_{0, 0, k}^{(s)}, \ \mathcal{S}_{0, k}^{(s)}, \\
& \mathcal{S}_{k, 0}^{(s)}, \ \mathcal{R}_{k, 0, 0}^{(s)}, \ \mathcal{R}_{0, 1, 0}^{(s)}, \ \mathcal{R}_{0, 2, 0}^{(s)}, \ U_{k, 0}^{(s)}, \ V_{k, 0}^{(s)}, \ V_{0, k}^{(s)},
\end{align*}
are Kirillov-Reshetikhin modules, they are special. In the following, we will prove that the other modules in (\ref{need to prove to be special}) are special. Without loss of generality, we may assume that $s=0$ in $\mathcal{T}^{(s)}$, where $\mathcal{T}$ is a module in \ref{need to prove to be special}.

\subsection{The case of $\mathcal{T}_{k, \ell, 0}^{(0)}$} \label{The case of Tkl0}

Let $m_+=T_{k, \ell, 0}^{(0)}$. Then
\begin{align*}
m_+=(3_{0}3_{4}\cdots 3_{4k-4})(2_{4k+1}2_{4k+3}\cdots 2_{4k+2\ell-1}).
\end{align*}

Let
$$U=I \times \{aq^s  : s \in \mathbb{Z},   s \leq 4k+2\ell-1 \}. $$
Since all monomials in $\mathscr{M}(\chi_q(m_+)-\text{trunc}_{m_+ \mathcal{Q}_{U}^{-}} \: \chi_q(m_+))$ are right-negative, it is sufficient to show that $\text{trunc}_{m_+ \mathcal{Q}_{U}^{-}} \: \chi_q(m_+)$ is special.

Let
\begin{align*}
\mathcal{M} = \{m_+\prod_{j=0}^{s-1} A_{3, 4k-4j-2}^{-1}: 0 \leq s \leq k-1\}.
\end{align*}

It is easy to see that $\mathcal{M}$ satisfies the conditions in Theorem \ref{truncated}. Therefore
\begin{align*}
\text{trunc}_{m_+ \mathcal{Q}_{U}^{-}} \: \chi_q(m_+)=\sum_{m\in \mathcal{M}} m
\end{align*}
and hence $\text{trunc}_{m_+ \mathcal{Q}_{U}^{-}} \: \chi_q(m_+)$ is special.

\subsection{Some other cases in \ref{need to prove to be special}} \label{table}
The special property of many modules in \ref{need to prove to be special} is proved using the same arguments as  Section \ref{The case of Tkl0} by using Theorem \ref{truncated} and appropriate $U$ and $\mathcal{M}$. We list the modules and corresponding $m_+$, $U$, and $\mathcal{M}$ in the following table. 

{\tiny
\begin{table}[h]
\begin{tabular}{|c|c|c|c|}
\hline %
module & $m_{+}$ & $U$ & monomials in $\mathcal{M}$ \\ 
\hline %
$\mathcal{T}_{1, 0, m}^{(0)}$ & $ 3_0 \left( \prod_{j=0}^{m-1} 1_{4k+2j+2} \right)$ & $I \times \{aq^s : s \in \mathbb{Z}, s \leq 2m+4 \}$ & $\substack{m_0=m_+, m_1=m_0A_{3,2}^{-1}, \\ m_2=m_1A_{2,4}^{-1}, m_3=m_2A_{2,2}^{-1}, \\ 
m_4=m_3A_{3,4}^{-1}}$ \\
\hline %
$\tilde{\mathcal{T}}_{1, 0, m}^{(0)}$ & $ 1_0 \left( \prod_{j=0}^{m-1} 3_{4j+6} \right)$ & $I \times \{aq^s : s \in \mathbb{Z}, s \leq 4m+2 \}$ & $\substack{m_0=m_+,  m_1=m_0A_{1,1}^{-1}, \\ m_2=m_1A_{2,2}^{-1}}$ \\
\hline %
$\mathcal{T}_{0, \ell, 1}^{(0)}$ & $ \left( \prod_{j=0}^{k-1} 2_{2j+1} \right)1_{2k+2}$ & $I \times \{aq^s : s \in \mathbb{Z}, s \leq 2k+2 \}$ & $m_+ \prod_{j=0}^{s-1} A_{2, 2(k-j)}, 0 \leq s \leq k$ \\
\hline %
$\mathcal{T}_{0, \ell, 2}^{(0)}$  & $\left( \prod_{j=0}^{k-1} 2_{2j+1} \right) 1_{2k+2} 1_{2k+4}$ & $I \times \{aq^s : s \in \mathbb{Z}, s \leq 2k+4 \}$ & $ \substack{ m_+ \prod_{j=0}^{s-1} A_{2, 2(k-j)}, 0 \leq s \leq k,  \\ n_1 = m_+ A_{2, 2k}^{-1} A_{3, 2k+2}^{-1}, \\ n_1 \prod_{j=1}^{s_1-1} A_{2, 2(k-j)}, 2\leq s_1 \leq k }$ \\ 
\hline %
$\mathcal{S}_{1, \ell}^{(0)}$  & $2_0 \left( \prod_{j=0}^{\ell-1} 2_{2j+6} \right) $ & $I \times \{aq^s : s \in \mathbb{Z}, s \leq 2\ell+4 \}$ & $ \substack{ m_0=m_+, \\ m_1=m_0 A_{2,2}^{-1}, m_2=m_1A_{1,2}^{-1} \\ m_3 = m_1 A_{3,3}^{-1}, m_4 = m_3A_{1,2}^{-1} }$ \\ 
\hline %
$\mathcal{R}_{k, 2\ell, \ell}^{(0)}$  & $\substack{\left( \prod_{j=0}^{k-1} 2_{2j} \right) \times \\ \times \left( \prod_{j=0}^{2\ell-1} 1_{2k+2j+1} \right) \left( \prod_{j=0}^{\ell-1} 3_{2k+4j+3} \right) }$ & $I \times \{aq^s : s \in \mathbb{Z}, s \leq 2k+4\ell-1 \}$ & $ m_+ \prod_{j=0}^{s-1} A_{2,2k-2j-1}^{-1}, 0 \leq s \leq k-1 $ \\ 
\hline %
$\mathcal{R}_{k, 2\ell+1, \ell}^{(0)}$  & $\substack{\left( \prod_{j=0}^{k-1} 2_{2j} \right)  \times \\ \times \left( \prod_{j=0}^{2\ell} 1_{2k+2j+1} \right) \left( \prod_{j=0}^{\ell-1} 3_{2k+4j+3} \right) }$ & $I \times \{aq^s : s \in \mathbb{Z}, s \leq 2k+4\ell+1 \}$ & $ m_+ \prod_{j=0}^{s-1} A_{2,2k-2j-1}^{-1}, 0 \leq s \leq k-1 $ \\ 
\hline %
$\mathcal{R}_{k, 2\ell+2, \ell}^{(0)}$  & $\substack{\left( \prod_{j=0}^{k-1} 2_{2j} \right) \times \\ \times \left( \prod_{j=0}^{2\ell+1} 1_{2k+2j+1} \right) \left( \prod_{j=0}^{\ell-1} 3_{2k+4j+3} \right) }$ & $I \times \{aq^s : s \in \mathbb{Z}, s \leq 2k+4\ell+3 \}$ & $ \substack{ m_+, \\ m_s = m_+ \prod_{j=1}^s A_{2,2(k-s)+1}^{-1}, 1 \leq s \leq k \\ m_s \prod_{t=0}^{j-1} A_{3, 2k+4\ell - 4t +1}^{-1}, 1\leq j \leq \ell } $ \\ 
\hline %
$\mathcal{R}_{0, 2\ell, \ell}^{(0)}$  & $  \left( \prod_{j=0}^{2\ell-1} 1_{2j+1} \right) \left( \prod_{j=0}^{\ell-1} 3_{4j+3} \right) $ & $I \times \{aq^s : s \in \mathbb{Z}, s \leq 4\ell-1 \}$ & $ m_+ $ \\ 
\hline %
$\mathcal{R}_{0, 2\ell+1, \ell}^{(0)}$  & $ \left( \prod_{j=0}^{2\ell} 1_{ 2j+1} \right) \left( \prod_{j=0}^{\ell-1} 3_{ 4j+3} \right) $ & $I \times \{aq^s : s \in \mathbb{Z}, s \leq  4\ell+1 \}$ & $ m_+ $ \\ 
\hline %
$\mathcal{R}_{0, 2\ell+2, \ell}^{(0)}$  & $ \left( \prod_{j=0}^{2\ell+1} 1_{ 2j+1} \right) \left( \prod_{j=0}^{\ell-1} 3_{ 4j+3} \right) $ & $I \times \{aq^s : s \in \mathbb{Z}, s \leq 4\ell+3 \}$ & $ m_{+} \prod_{t=0}^{j-1} A_{3,  4\ell - 4t +1}^{-1}, 0\leq j \leq \ell $ \\ 
\hline %
$\mathcal{U}_{k, \ell}^{(0)}$  & $ \left( \prod_{j=0}^{k-1} 2_{ 2j} \right) \left( \prod_{j=0}^{\ell-1} 3_{ 2k+2j+1} \right) $ & $I \times \{aq^s : s \in \mathbb{Z}, s \leq 2k+2\ell-1 \}$ & $ \substack{ m_s = m_+ \prod_{j=1}^s A_{2,2(k-s)+1}^{-1}, 0\leq s \leq k, \\ m_{js} = m_s \prod_{t=0}^{j-1} A_{1,2(k-j)}^{-1}, 1\leq j \leq s } $ \\ 
\hline %
$\mathcal{V}_{1, \ell}^{(0)}$  & $ 3_0 \left( \prod_{j=0}^{\ell-1} 3_{ 4j+6} \right) $ & $I \times \{aq^s : s \in \mathbb{Z}, s \leq 4\ell+2 \}$ & $ \substack{ m_0=m_+, m_1 = m_0 A_{3,2}^{-1},  \\ 
m_2 = m_1 A_{2,4}^{-1}, m_3 = m_2 A_{2,2}^{-1}, \\
m_4 = m_2 A_{1,5}^{-1}, m_5 = m_4 A_{2,2}^{-1}, \\
m_6 = m_5 A_{1,3}^{-1} } $ \\ 
\hline %
$\mathcal{P}_{0, \ell}^{(0)}$  & $ 2_1 \left( \prod_{j=0}^{\ell-1} 3_{ 4j+6} \right) $ & $I \times \{aq^s : s \in \mathbb{Z}, s \leq 4\ell+2 \}$ & $ \substack{ m_0 = m_+,  m_1 = m_0 A_{2,2}^{-1}, \\ m_2 = m_1 A_{1,3}^{-1} } $ \\ 
\hline %
\end{tabular}
\caption{some modules which are special}
\label{modules which are special}
\end{table}
}
By using the same arguments in Section \ref{The case of Tkl0}, we can prove that the modules in the above table are special.

\subsection{The case of $\mathcal{T}_{k, 0, m}^{(0)}$}

Let $m_+=T_{k, 0, m}^{(0)}$. Then
\begin{align*}
m_+=(3_{0}3_{4}\cdots 3_{4k-4})(1_{4k+2}1_{4k+4}\cdots 1_{4k+2\ell}).
\end{align*}

If $k=1$, then $\mathcal{T}_{k, 0, m}^{(0)}$ is special by the result of Section \ref{table}.

\textbf{Case 2.} $k>1$. We embed $L(m_+)$ into two different tensor products. Since each factor in the tensor product is special, we can use the FM algorithm to compute the $q$-characters of the factors. We classify the dominant monomials in the first tensor product and prove that the only dominant monomial in the first tensor product which occurs in the second tensor product is $m_+$. Hence $L(m_+)$ is special.

The first tensor product is $L(m'_1) \otimes L(m'_2)$, where
\begin{align*}
m'_1=3_{0}3_{4}\cdots 3_{4k-8}, \ m'_2=3_{4k-4}1_{4k+2}1_{4k+4}\cdots 1_{4k+2\ell}.
\end{align*}

In Case 1, we have shown that $L(m'_2)$ is special. Therefore the FM algorithm works for $L(m'_2)$. We will use the FM algorithm to compute $\chi_q(L(m'_1)),  \chi_q(L(m'_2))$ and classify all dominant monomials in $\chi_q(L(m'_1))  \chi_q(L(m'_2))$. Let $m=m_1m_2$ be a dominant monomial, where $m_i \in \mathscr{M}(L(m'_i))$, $i=1, 2$.

Suppose that $m_2 \neq m'_2$. If $m_2$ is right-negative, then $m$ is a right negative monomial and therefore $m$ is not dominant. This is a contradiction. Hence $m_2$ is not right-negative. By Case 1, $m_2$ is one of the following monomials
\begin{align*}
& \bar{m}_1=m'_2 A_{3, 4k-2}^{-1} = 3_{4k}^{-1} 2_{4k-3} 2_{4k-1} 1_{4k+2}1_{4k+4}\cdots 1_{4k+2\ell}, \\
& \bar{m}_2=\bar{m}_1A_{2, 4k}^{-1} = 2_{4k-3} 2_{4k+1}^{-1} 1_{4k} 1_{4k+2}1_{4k+4}\cdots 1_{4k+2\ell}, \\
& \bar{m}_{3}=\bar{m}_{2}A_{2, 4k-2}^{-1} = 2_{4k-1}^{-1} 2_{4k+1}^{-1} 3_{4k-2} 1_{4k-2} 1_{4k} 1_{4k+2}1_{4k+4}\cdots 1_{4k+2\ell}, \\
& \bar{m}_{4}=\bar{m}_{3}A_{3, 4k}^{-1} = 3_{4k+2}^{-1} 1_{4k-2} 1_{4k} 1_{4k+2}1_{4k+4}\cdots 1_{4k+2\ell}.
\end{align*}
By Lemma \ref{fundamental q-characters}, $3_{4k}^{-1}$ cannot be canceled by any monomial in $\chi_q(m'_1)$. Therefore $m=m_1m_2$ $(m_1 \in \chi_q(m'_1))$ is not dominant. This is a contradiction. Hence $m_2 \neq \bar{m}_1$. Similarly, $m_2$ cannot be $\bar{m}_i, i=2, 3, 4$. This is a contradiction. Therefore $m_2=m'_2$.

If $m_1 \neq m'_1$, then $m_1$ is right negative. Since $m$ is dominant, each factor with a negative power in $m_1$ needs to be canceled by a factor in $m'_2$. By Lemma \ref{fundamental q-characters}, the only factor in $m'_2$ which can be canceled is $3_{4k-4}$. We have $\mathcal{M}(L(m'_1)) \subset \mathcal{M}(\chi_q(3_{0}3_{4}\cdots 3_{4k-12})) \chi_q(L(3_{4k-8})))$. Only monomials in $\chi_q(L(3_{4k-8}))$ can cancel $3_{4k-4}$. The only monomial in  $\chi_q(L(3_{4k-8}))$ which can cancel $3_{4k-4}$ is $3_{4k-4}^{-1}2_{4k-7}2_{4k-5}$. Therefore
$m_1$ is in the set
\begin{align*}
\mathcal{M}(\chi_q(3_{0}3_{4}\cdots 3_{4k-12}))3_{4k-4}^{-1}2_{4k-7}2_{4k-5}.
\end{align*}

If $m_1 = ( 3_{0}3_{4}\cdots 3_{4k-12} ) 3_{4k-4}^{-1}2_{4k-7}2_{4k-5}$, then
\begin{align}
m=m_1m_2=3_{0}3_{4}\cdots 3_{4k-12} 2_{4k-7}2_{4k-5} 1_{4k+2}1_{4k+4}\cdots 1_{4k+2\ell} \label{the other dominant monomial}
\end{align}
is dominant. Suppose that
\begin{align*}
m_1 \neq ( 3_{0}3_{4}\cdots 3_{4k-12} ) 3_{4k-4}^{-1}2_{4k-7}2_{4k-5}.
\end{align*}
Then $m_1=n_1 3_{4k-4}^{-1}2_{4k-7}2_{4k-5}$, where $n_1$ is a non-highest monomial in $\chi_q(3_{0}3_{4}\cdots 3_{4k-12})$. Since $n_1$ is right negative, $2_{4k-7}$ or $2_{4k-5}$ should cancel a factor of $n_1$ with a negative power. It is easy to see that there exists either a factor $2_{4k-9}^2$ or $2_{4k-7}^2$ in a monomial in $\chi_q(3_{0}3_{4}\cdots 3_{4k-12})3_{4k-4}^{-1}2_{4k-7}2_{4k-5}$ by using the FM algorithm. Therefore we need a factor $2_{4k-9}$ or $2_{4k-7}$ in a monomial in $\chi_q(3_{0}3_{4}\cdots 3_{4k-12})$. We have
\begin{align*}
\chi_q(3_{0}3_{4}\cdots 3_{4k-12})3_{4k-4}^{-1}2_{4k-7}2_{4k-5} \subseteq \chi_q(3_{0}3_{4}\cdots 3_{4k-16})\chi_q(3_{4k-12})3_{4k-4}^{-1}2_{4k-7}2_{4k-5}.
\end{align*}
The factors $2_{4k-9}$ and $2_{4k-7}$ can only come from the monomials in $\chi_q(3_{4k-12})$. By Lemma \ref{fundamental q-characters}, any monomial in $\chi_q(3_{4k-12})$ does not have a factor $2_{4k-9}$. The only monomial in $\chi_q(3_{4k-12})$ which contains a factor $2_{4k-7}$ is $1_{4k-10}1_{4k-6}^{-1} 2_{4k-7} 3_{4k-6}^{-1}$. Therefore $m_1$ is in the set
\begin{eqnarray*}
&& \mathcal{M}(\chi_q(3_{0}3_{4}\cdots 3_{4k-16}))1_{4k-10}1_{4k-6}^{-1} 2_{4k-7} 3_{4k-6}^{-1} 3_{4k-4}^{-1}2_{4k-7}2_{4k-5} \\
& = & \mathcal{M}(\chi_q(3_{0}3_{4}\cdots 3_{4k-16}))1_{4k-10}1_{4k-6}^{-1} 3_{4k-6}^{-1} 3_{4k-4}^{-1} 2_{4k-7}^2 2_{4k-5}.
\end{eqnarray*}
Since $m=m_1m_2$ is dominant, $1_{4k-6}^{-1} 3_{4k-6}^{-1}$ should be canceled by some monomial in $\chi_q(3_{0}3_{4}\cdots 3_{4k-16})$. But by Lemma \ref{fundamental q-characters}, $1_{4k-6}^{-1} 3_{4k-6}^{-1}$ cannot be canceled by any monomial in $\chi_q(3_{0}3_{4}\cdots 3_{4k-16})$. This is a contradiction.

Therefore the only dominant monomials in $\chi_q(L(m'_1))  \chi_q(L(m'_2))$ are $m_+$ and (\ref{the other dominant monomial}).

The second tensor product is $L(m''_1) \otimes L(m''_2)$, where
\begin{align*}
m''_1=3_{0}3_{4}\cdots 3_{4k-4}, \ m''_2=1_{4k+2}1_{4k+4}\cdots 1_{4k+2\ell}.
\end{align*}

The monomial (\ref{the other dominant monomial}) is
\begin{align}
n=m_+ A_{3, 4k-2}^{-1}. \label{expression of n}
\end{align}
Since $A_{i, a}, i\in I, a\in \mathbb{C}^{\times}$ are algebraically independent, the expression (\ref{expression of n}) of $n$ of the form $m_+\prod_{i\in I, a\in \mathbb{C}^{\times}} A_{i, a}^{-v_{i, a}}$, where $v_{i, a}$ are some integers, is unique. Suppose that the monomial $n$ is in $\chi_q(L(m''_1))  \chi_q(L(m''_2))$. Then $n=n_1n_2$, where $n_i \in \mathscr{M}(L(m''_i)), i=1, 2$. By the expression (\ref{expression of n}), we have $n_2=m''_2$ and $n_1=m''_1A_{3, 4k-2}^{-1}$. By the FM algorithm, the monomial $m''_1A_{3, 4k-2}^{-1}$ is not in $\mathscr{M}(L(m''_1))$. This contradicts the fact that $n_1 \in \mathscr{M}(L(m''_1))$. Therefore $n$ is not in $\chi_q(L(m''_1))  \chi_q(L(m''_2))$.

\subsection{The case of $\mathcal{\tilde{T}}_{k, 0, m}^{(0)}$}

Let $m_+=\tilde{T}_{k, 0, m}^{(0)}$ with $k, m \in\mathbb{Z}_{\geq 0}$. Then
\begin{align*}
m_+=(1_{0}1_{2}\cdots 1_{2k-2}) ( 3_{2k+4} 3_{2k+8} \cdots 3_{2k+4m} )
\end{align*}

If $k=1$, then $\mathcal{\tilde{T}}_{1, 0, m}^{(0)}$ is special by the result of Section \ref{table}.

\textbf{Case 2.} $k>1$. Let
\begin{align*}
& m'_1=1_{0}1_{2}\cdots 1_{2k-4}, \ m'_2=1_{2k-2} 3_{2k+4} 3_{2k+8} \cdots 3_{2k+4m}, \\
& m''_1=1_{0}1_{2}\cdots 1_{2k-2}, \ m''_2=3_{2k+4} 3_{2k+8} \cdots 3_{2k+4m}.
\end{align*}
Then $\mathscr{M}(L(m_+)) \subset \mathscr{M}(\chi_q(m'_1)\chi_q(m'_2)) \cap \mathscr{M}(\chi_q(m''_1)\chi_q(m''_2))$.

By using similar arguments as the case of $\mathcal{T}_{k, 0, m}^{(s)}$, we show that the only possible dominant monomials in $\chi_q(m'_1)\chi_q(m'_2)$ are $m_+$ and
\begin{eqnarray*}
n_1=1_{0}1_{2}\cdots 1_{2k-6} 2_{2k-3} 3_{2k+4} 3_{2k+8} \cdots 3_{2k+4m} = m_+ A_{1, 2k-3}^{-1}.
\end{eqnarray*}
Moreover, $n_1$ is not in $\chi_q(m''_1)\chi_q(m''_2)$. Therefore the only dominant monomial in $\chi_q(m_+)$ is $m_+$.

\subsection{The case of $\mathcal{S}_{k, \ell}^{(0)}$}

Let $m_+=S_{k, \ell}^{(0)}$ with $k, \ell \in\mathbb{Z}_{\geq 0}$. Then
\begin{align*}
m_+=(2_{0}2_{2}\cdots 2_{2k-2}) ( 2_{2k+4} 2_{2k+6} \cdots 2_{2k+2\ell+2} ).
\end{align*}

If $k=1$, then $\mathcal{S}_{1, \ell}^{(s)}$ is special by the result of Section \ref{table}.

\textbf{Case 2.} $k>1$. Let
\begin{align*}
& m'_1=2_{0}2_{2}\cdots 2_{2k-4}, \ m'_2=2_{2k-2} 2_{2k+4} 2_{2k+6} \cdots 2_{2k+2\ell+2}, \\
& m''_1=2_{0}2_{2}\cdots 2_{2k-2}, \ m''_2=2_{2k+4} 2_{2k+6} \cdots 2_{2k+2\ell+2}.
\end{align*}
Then $\mathscr{M}(L(m_+)) \subset \mathscr{M}(\chi_q(m'_1)\chi_q(m'_2)) \cap \mathscr{M}(\chi_q(m''_1)\chi_q(m''_2))$.

By using similar arguments as the case of $\mathcal{T}_{k, 0, m}^{(s)}$, we show that the only possible dominant monomials in $\chi_q(m'_1)\chi_q(m'_2)$ are $m_+$ and
\begin{eqnarray*}
n_1 & = & 2_{0}2_{2}\cdots 2_{2k-6} 1_{2k-3} 3_{2k-3} 2_{2k+4} 2_{2k+6} \cdots 2_{2k+2\ell+2} = m_+ A_{2, 2k-3}^{-1}, \\
n_2 & = & 2_{0}2_{2}\cdots 2_{2k-6} 2_{2k+4} 2_{2k+8} 2_{2k+10} \cdots 2_{2k+2\ell+2} \\
& = & m_+ A_{2, 2k-3}^{-1} A_{1, 2k-2}^{-1} A_{3, 2k-1}^{-1} A_{2, 2k+1}^{-1} A_{2, 2k-1}^{-1} A_{3, 2k+1}^{-1} A_{1, 2k+2}^{-1} A_{2, 2k+3}^{-1}, \\
n_3 & = & 2_{0}2_{2}\cdots 2_{2k-10} 1_{2k-7} 3_{2k-5} 2_{2k+8} 2_{2k+10} \cdots 2_{2k+2\ell+2} \\
& = & n_2 A_{2, 2k-5}^{-1} A_{2, 2k-7}^{-1} A_{1, 2k-4}^{-1} A_{3, 2k-5}^{-1} A_{2, 2k-3}^{-1}.
\end{eqnarray*}
Moreover, $n_1$, $n_2$, $n_3$ are not in $\chi_q(m''_1)\chi_q(m''_2)$. Therefore the only dominant monomial in $\chi_q(m_+)$ is $m_+$.

\subsection{The case of $\mathcal{U}_{k, \ell}^{(s)}$}

Let $m_+=U_{k, \ell}^{(s)}$ with $k, \ell \in\mathbb{Z}_{\geq 0}$. Without loss of generality, we may assume that $s=0$. Then
\begin{align*}
m_+=(2_{0}2_{2}\cdots 2_{2k-2}) ( 3_{2k+1} 3_{2k+3} \cdots 3_{2k+2\ell-1} ).
\end{align*}

Let
$$
U=I \times \{aq^s :  s\in \mathbb{Z},  s \leq 2k+2\ell-1 \}.
$$

Let $\mathcal{M}$ be the finite set consisting of the following monomials
\begin{align*}
& m_+, \ m_{1}=m_{+}A_{2, 2k-1}^{-1}, \ m_{2}=m_{1} A_{2, 2k-3}^{-1}, \ldots, m_{k}=m_{k-1} A_{2, 1}^{-1}, \\
& m_{11}=m_{1}A_{1, 2k}^{-1}, \\
& m_{12}=m_{2}A_{1, 2k}^{-1}, \ m_{22}=m_{12} A_{1, 2k-2}^{-1}, \\
& \vdots \\
& m_{1k}=m_{k}A_{1, 2k}^{-1}, \ m_{2k}=m_{1k} A_{1, 2k-2}^{-1}, \ldots, m_{kk}=m_{k-1, k} A_{1, 2}^{-1}.
\end{align*}
By Theorem \ref{truncated},
\begin{align*}
\text{trunc}_{m_+ \mathcal{Q}_{U}^{-}} \: \chi_q(m_+)=\sum_{m\in \mathcal{M}} m
\end{align*}
and hence $\text{trunc}_{m_+ \mathcal{Q}_{U}^{-}} \: \chi_q(m_+)$ is special. Therefore $\mathcal{U}_{k, \ell}^{(s)}$ is special.

\subsection{The case of $\mathcal{V}_{k, \ell}^{(s)}$}
Let $m_+=V_{k, \ell}^{(s)}$ with $k, \ell \in\mathbb{Z}_{\geq 0}$. Without loss of generality, we may assume that $s=0$. Then
\begin{align*}
m_+=(3_{0}3_{4}\cdots 3_{4k-4}) ( 3_{4k+2} 3_{4k+6} \cdots 3_{4k+4\ell-2} ).
\end{align*}

\textbf{Case 1.} $k=1$. In this case,
\begin{align*}
m_+ = 3_{0} (3_{6} 3_{10} \cdots 3_{4\ell+2}).
\end{align*}

Let
$$
U=I \times \{aq^s :  s\in \mathbb{Z},  s \leq 4\ell+2 \}.
$$

Let $\mathcal{M}$ be the finite set consisting of the following monomials
\begin{align*}
& m_0=m_+, \ m_1=m_0A_{3, 2}^{-1}, \ m_2=m_1A_{2, 4}^{-1}, \\
& m_3=m_2A_{2, 2}^{-1}, \ m_4=m_2A_{1, 5}^{-1}, \ m_5=m_4A_{2, 2}^{-1}, \ m_6=m_5A_{1, 3}^{-1}.
\end{align*}
By Theorem \ref{truncated},
\begin{align*}
\text{trunc}_{m_+ \mathcal{Q}_{U}^{-}} \: \chi_q(m_+)=\sum_{m\in \mathcal{M}} m
\end{align*}
and hence $\text{trunc}_{m_+ \mathcal{Q}_{U}^{-}} \: \chi_q(m_+)$ is special. Therefore $\mathcal{V}_{1, \ell}^{(s)}$ is special.

\textbf{Case 2.} $k>1$. Let
\begin{align*}
& m'_1=3_{0}3_{4}\cdots 3_{4k-8}, \ m'_2= 3_{4k-4} 3_{4k+2} 3_{4k+6} \cdots 3_{4k+4\ell-2}, \\
& m''_1=3_{0}3_{4}\cdots 3_{4k-4}, \ m''_2=3_{4k+2} 3_{4k+6} \cdots 3_{4k+4\ell-2}.
\end{align*}
Then $\mathscr{M}(L(m_+)) \subset \mathscr{M}(\chi_q(m'_1)\chi_q(m'_2)) \cap \mathscr{M}(\chi_q(m''_1)\chi_q(m''_2))$.

By using similar arguments as the case of $\mathcal{T}_{k, 0, m}^{(s)}$, we show that the only possible dominant monomials in $\chi_q(m'_1)\chi_q(m'_2)$ are $m_+$ and
\begin{eqnarray*}
n_1=(3_{0}3_{4}\cdots 3_{4k-12}) 2_{4k-7} 2_{4k-5} (3_{4k+2} 3_{4k+6} \cdots 3_{4k+4\ell-2})=m_+ A_{3, 4k-6}^{-1}.
\end{eqnarray*}
Moreover, $n_1$ is not in $\chi_q(m''_1)\chi_q(m''_2)$. Therefore the only dominant monomial in $\chi_q(m_+)$ is $m_+$.

\subsection{The case of $\mathcal{P}_{k, \ell}^{(s)}$}
Let $m_+=P_{k, \ell}^{(s)}$ with $k, \ell \in\mathbb{Z}_{\geq 0}$. Without loss of generality, we may assume that $s=0$. Then
\begin{align*}
m_+=(3_{0}3_{4}\cdots 3_{4k-4}) 2_{4k+1} ( 3_{4k+6} 3_{4k+10} \cdots 3_{4k+4\ell+2} ).
\end{align*}

\textbf{Case 1.} $k=0$. In this case,
\begin{align*}
m_+ = 2_{1} (3_{6} 3_{10} \cdots 3_{4\ell+2}).
\end{align*}

Let
$$
U=I \times \{aq^s :  s\in \mathbb{Z},  s \leq 4\ell+2 \}.
$$

Let $\mathcal{M}$ be the finite set consisting of the following monomials
\begin{align*}
& m_0=m_+, \ m_1=m_0A_{2, 2}^{-1}, \ m_2=m_1A_{1, 3}^{-1}.
\end{align*}
It is clear that $\mathcal{M}$ satisfies the conditions in Theorem \ref{truncated}. Therefore
\begin{align*}
\text{trunc}_{m_+ \mathcal{Q}_{U}^{-}} \: \chi_q(m_+)=\sum_{m\in \mathcal{M}} m
\end{align*}
and hence $\text{trunc}_{m_+ \mathcal{Q}_{U}^{-}} \: \chi_q(m_+)$ is special. Therefore $\mathcal{P}_{0, \ell}^{(s)}$ is special.

\textbf{Case 2.} $k>0$. Let
\begin{align*}
& m'_1=3_{0}3_{4}\cdots 3_{4k-4}, \ m'_2= 2_{4k+1} 3_{4k+6} 3_{4k+10} \cdots 3_{4k+4\ell+2}, \\
& m''_1=3_{0}3_{4}\cdots 3_{4k-4} 2_{4k+1}, \ m''_2= 3_{4k+6} 3_{4k+10} \cdots 3_{4k+4\ell+2}.
\end{align*}
Then $\mathscr{M}(L(m_+)) \subset \mathscr{M}(\chi_q(m'_1)\chi_q(m'_2)) \cap \mathscr{M}(\chi_q(m''_1)\chi_q(m''_2))$. Note that we have shown that $m''_2=T_{k, 1, 0}^{(0)}$ is special in the case of $T_{k, \ell, 0}^{(s)}$. Therefore the FM algorithm applies to $m''_2$.

By using similar arguments as the case of $\mathcal{T}_{k, 0, m}^{(s)}$, we show that the only possible dominant monomials in $\chi_q(m'_1)\chi_q(m'_2)$ are $m_+$ and
\begin{eqnarray*}
n_1=(3_{0}3_{4}\cdots 3_{4k-8}) 2_{4k-3} 1_{4k} (3_{4k+6} 3_{4k+10} \cdots 3_{4k+4\ell+2}) = m_+ A_{3, 4k-2}^{-1} A_{2, 4k}^{-1}.
\end{eqnarray*}
Moreover, $n_1$ is not in $\chi_q(m''_1)\chi_q(m''_2)$. Therefore the only dominant monomial in $\chi_q(m_+)$ is $m_+$.

\subsection{The case of $\mathcal{O}_{k, \ell}^{(s)}$}
Let $m_+=O_{k, \ell}^{(s)}$ with $k, \ell \in\mathbb{Z}_{\geq 0}$. Without loss of generality, we may assume that $s=0$. Then
\begin{align*}
m_+=(2_{0}2_{2}\cdots 2_{2k-2}) 1_{2k+1} 1_{2k+3} ( 2_{2k+6} 2_{2k+8} \cdots 2_{2k+2\ell+4} ).
\end{align*}

\textbf{Case 1.} $k=0$. In this case,
\begin{align*}
m_+ = 1_{1} 1_{3} (2_{6} 2_{8} \cdots 2_{2\ell+4}) = \tilde{T}_{2, \ell, 0}^{(1)}.
\end{align*}
By Theorem \ref{special already known}, $L(m_+)=\mathcal{O}_{0, \ell}^{(s)}$ is special.

\textbf{Case 2.} $k>0$. Let
\begin{align*}
& m'_1=2_{0}2_{2}\cdots 2_{2k-2}, \ m'_2= 1_{2k+1} 1_{2k+3} 2_{2k+6} 2_{2k+8} \cdots 2_{2k+2\ell+4}, \\
& m''_1=2_{0}2_{2}\cdots 2_{2k-2}1_{2k+1}1_{2k+3}, \ m''_2= 2_{2k+6} 2_{2k+8} \cdots 2_{2k+2\ell+4}.
\end{align*}
Then $\mathscr{M}(L(m_+)) \subset \mathscr{M}(\chi_q(m'_1)\chi_q(m'_2)) \cap \mathscr{M}(\chi_q(m''_1)\chi_q(m''_2))$. Note that we have shown that $m''_2=T_{0, k, 2}^{(0)}$ is special in the case of $T_{0, k, r}^{(s)}$, $r\in \{1, 2\}$. Therefore the FM algorithm applies to $m''_2$.

By using similar arguments as the case of $\mathcal{T}_{k, 0, m}^{(s)}$, we show that the only possible dominant monomials in $\chi_q(m'_1)\chi_q(m'_2)$ are $m_+$ and
\begin{eqnarray*}
n_1 & = & (2_{0}2_{2}\cdots 2_{2k-4}) 3_{2k-1} 1_{2k+3} (2_{2k+6} 2_{2k+8} \cdots 2_{2k+2\ell+4}) \\
& = & m_+ A_{2, 2k-1}^{-1} A_{1, 2k}^{-1}, \\
n_2 & = & (2_{0}2_{2}\cdots 2_{2k-4}) 1_{2k+1} 1_{2k+3} (2_{2k+8} 2_{2k+10} \cdots 2_{2k+2\ell+4}) \\
& = & n_1 A_{3, 2k+1}^{-1} A_{2, 2k+3}^{-1} A_{2, 2k+1}^{-1} A_{3, 2k+3}^{-1} A_{1, 2k+4}^{-1} A_{2, 2k+5}^{-1}, \\
n_3 & = & (2_{0}2_{2}\cdots 2_{2k-6}) 1_{2k-3} 1_{2k+1} (2_{2k+8} 2_{2k+10} \cdots 2_{2k+2\ell+4}) \\
& = & n_2 A_{2, 2k-3}^{-1} A_{3, 2k-1}^{-1} A_{2, 2k+1}^{-1} A_{1, 2k+2}^{-1}, \\
n_4 & = & (2_{0}2_{2}\cdots 2_{2k-8}) 1_{2k-5} 1_{2k-3} (2_{2k+8} 2_{2k+10} \cdots 2_{2k+2\ell+4}) \\
& = & n_3 A_{2, 2k-5}^{-1} A_{3, 2k-3}^{-1} A_{2, 2k-1}^{-1} A_{1, 2k}^{-1}.
\end{eqnarray*}
Moreover, $n_1$, $n_2$, $n_3$, $n_4$ are not in $\chi_q(m''_1)\chi_q(m''_2)$. Therefore the only dominant monomial in $\chi_q(m_+)$ is $m_+$.

\section{Proof Theorem \ref{the system I} and Theorem \ref{the system III}} \label{prove system}
In this section, we will prove Theorem \ref{the system I} and Theorem \ref{the system III}. We use the FM algorithm to classify dominant monomials in $\chi_q(\mathcal{L})\chi_q(\mathcal{R})$, $\chi_q(\mathcal{T})\chi_q(\mathcal{B})$, and $\chi_q(\mathcal{S})$.

\subsection{Classification of dominant monomials in $\chi_q(\mathcal{L})\chi_q(\mathcal{R})$ and $\chi_q(\mathcal{T})\chi_q(\mathcal{B})$}
\begin{lemma} \label{dominant monomials}
We have the following cases.
\begin{enumerate}[(1)]
\item Let $$M=T_{k, \ell-1, 0}^{(s)}T_{k-1, \ell}^{(s+4)}, k\geq 1, \ell \geq 1. $$ Then dominant monomials in $\chi_q(T_{k, \ell-1, 0}^{(s)} )  \chi_q( T_{k-1, \ell, 0}^{(s+4)})$ are
\begin{align*}
& M_0=M, \ M_1=MA_{2, s+4k+2\ell-2}^{-1}, \ M_2=M_1A_{2, s+4k+2\ell-4}^{-1}, \ \ldots, \nonumber \\ & M_{\ell-1}=M_{\ell-2}A_{2, s+4k+2}^{-1}, \ M_{\ell}=M_{\ell-1}A_{3, s+4k-2}^{-1}A_{2, s+4k}^{-1}, \nonumber \\ & M_{\ell+1}=M_{\ell}A_{3, s+4k-6}^{-1}, \ M_{\ell+2}=M_{\ell+1}A_{3, s+4k-10}^{-1}, \ \ldots, \ M_{k+\ell-1}=M_{k+\ell-2}A_{3, s+2}^{-1}.
\end{align*}
The dominant monomials in $\chi_q(T_{k, \ell, 0}^{(s)} )  \chi_q( T_{k-1, \ell-1, 0}^{(s+4)})$ are $M_0, \ldots, M_{k+\ell-2}$.

\item Let $$M=T_{k, 0, m-1}^{(s)}T_{k-1, 0, m}^{(s+4)}, k\geq 1, m \geq 1. $$ Then dominant monomials in $\chi_q( T_{k, 0, m-1}^{(s)} ) \chi_q( T_{k-1, 0, m}^{(s+4)} )$ are
\begin{align*}
& M_0=M, \ M_1=MA_{1, s+4k+2m-1}^{-1}, \ M_2=M_1A_{1, s+4k+2m-3}^{-1}, \ \ldots, \nonumber \\ & M_{m-1}=M_{m-2}A_{1, s+4k+3}^{-1}, \ M_{m}=M_{m-1}A_{3, s+4k-2}^{-1}A_{2, s+4k}^{-1}A_{1, s+4k+1}^{-1}, \nonumber \\ & M_{m+1}=M_{m}A_{3, s+4k-6}^{-1}, \ M_{m+2}=M_{m+1}A_{3, s+4k-10}^{-1}, \ \ldots, \ M_{k+m-1}=M_{k+m-2}A_{3, s+2}^{-1}.
\end{align*}
The dominant monomials in $\chi_q(T_{k, 0, m}^{(s)} )  \chi_q( T_{k-1, 0, m-1}^{(s+4)})$ are $M_0, \ldots, M_{k+m-2}$.

\item Let $$M=S_{k, \ell-1}^{(s)}S_{k-1, \ell}^{(s+2)}, k, \ell \geq 1.$$ Then dominant monomials in $\chi_q(S_{k, \ell-1}^{(s)} )  \chi_q( S_{k-1, \ell}^{(s+2)} )$ are
\begin{align*}
& M_0=M, \ M_1=MA_{2, s+2k+2\ell+1}^{-1}, \ M_2=M_1A_{2, s+2k+2\ell-1}^{-1}, \ \ldots,  \nonumber \\ &  M_{\ell-1}=M_{\ell-2}A_{2, s+2k+5}^{-1}, \ M_{\ell}=M_{\ell-1}A_{2, s+2k-1}^{-1}A_{3, s+2k+1}^{-1}A_{2, 2k+3}^{-1}, \ \nonumber \\ &  M_{\ell+1}=M_{\ell}A_{2, s+2k-3}^{-1}, \ M_{\ell+2}=M_{\ell+1}A_{2, s+2k-5}^{-1}, \ \ldots, M_{k+\ell-1}=M_{k+\ell-2}A_{2, s+1}^{-1}.
\end{align*}
The dominant monomials in $\chi_q( S_{k, \ell}^{(s)} ) \chi_q( S_{k-1, \ell-1}^{(s+2)} )$ are $M_0, \ldots, M_{k+\ell-2}$.

\item Let $$M=T_{0, \ell, r-1}^{(s)}T_{0, \ell-1, r}^{(s+2)}, \ell \geq 1, r \in \{1, 2\}. $$ Then dominant monomials in $\chi_q( T_{0, \ell, r-1}^{(s)} ) \chi_q( T_{0, \ell-1, r}^{(s+2)} )$ are
\begin{align*}
& M_0=M, \ M_1=MA_{1, s+2\ell+2r-2}^{-1}, \ M_2=M_1A_{1, s+2\ell+2r-4}^{-1}, \ \ldots, \nonumber \\ & M_{r-1}=M_{r-2}A_{1, s+2\ell+2}^{-1}, \ M_{r}=M_{r-1}A_{2, s+2\ell}^{-1}A_{1, s+2\ell+1}^{-1}, \nonumber \\ & M_{r+1}=M_{r}A_{2, s+2\ell-2}^{-1}, \ M_{r+2}=M_{r+1}A_{2, s+2\ell-4}^{-1}, \ \ldots, \ M_{\ell+r-1}=M_{\ell+r-2}A_{2, s+2}^{-1}.
\end{align*}
The dominant monomials in $\chi_q(T_{0, \ell, r}^{(s)} )  \chi_q( T_{0, \ell-1, r-1}^{(s+2)})$ are $M_0, \ldots, M_{\ell+r-2}$.

\item Let $$M=\tilde{T}_{k, 0, m-1}^{(s)}\tilde{T}_{k-1, 0, m}^{(s+2)}, k, m \geq 1. $$ Then dominant monomials in $\chi_q( \tilde{T}_{k, 0, m-1}^{(s)} ) \chi_q( \tilde{T}_{k-1, 0, m}^{(s+2)} )$ are
\begin{align*}
& M_0=M, \ M_1=MA_{3, s+2k+4m-2}^{-1}, \ M_2=M_1A_{3, s+2k+4m-6}^{-1}, \ \ldots, \nonumber \\ & M_{m-1}=M_{m-2}A_{3, s+2k+6}^{-1}, \ M_{m}=M_{m-1}A_{1, s+2k-1}^{-1}A_{2, s+2k}^{-1}A_{3, 2k+2}^{-1}, \nonumber \\ & M_{m+1}=M_{m}A_{1, s+2k-3}^{-1}, \ M_{m+2}=M_{m+1}A_{1, s+2k-5}^{-1}, \ \ldots, \ M_{k+m-1}=M_{k+m-2}A_{1, s+1}^{-1}.
\end{align*}
The dominant monomials in $\chi_q(\tilde{T}_{k, 0, m}^{(s)} )  \chi_q( \tilde{T}_{k-1, 0, m-1}^{(s+2)})$ are $M_0, \ldots, M_{\ell+m-2}$.

\item Let $$M=R_{k, 2\ell, \ell-1}^{(s)} R_{k-1, 2\ell, \ell}^{(s+2)}, k\geq 1, \ell \geq 1. $$ Then dominant monomials in $\chi_q( R_{k, 2\ell, \ell-1}^{(s)} ) \chi_q( R_{k-1, 2\ell, \ell}^{(s+2)} )$ are
\begin{align*}
& M_0=M, \ M_1=MA_{3, s+2k+4\ell-3}^{-1}, \ M_2=M_1A_{3, s+2k+4\ell-7}^{-1}, \ \ldots, \nonumber \\ & M_{\ell-1}=M_{\ell-2}A_{3, s+2k+5}^{-1}, \ M_{\ell}=M_{\ell-1}A_{2, s+2k-1}^{-1}A_{3, s+2k+1}^{-1}, \nonumber \\ & M_{\ell+1}=M_{\ell}A_{2, s+2k-3}^{-1}, \ M_{\ell+2}=M_{\ell+1}A_{2, s+2k-5}^{-1}, \ \ldots, \ M_{k+\ell-1}=M_{k+\ell-2}A_{2, s+1}^{-1}.
\end{align*}
The dominant monomials in $\chi_q(R_{k, 2\ell, \ell}^{(s)} )  \chi_q( R_{k-1, 2\ell, \ell-1}^{(s+2)})$ are $M_0, \ldots, M_{k+\ell-2}$.

\item Let $$M=R_{k, 2\ell, \ell}^{(s)} R_{k-1, 2\ell+1, \ell}^{(s+2)}, k\geq 1, \ell \geq 1. $$ Then dominant monomials in $\chi_q( R_{k, 2\ell, \ell}^{(s)} ) \chi_q( R_{k-1, 2\ell+1, \ell}^{(s+2)} )$ are
\begin{align*}
& M_0=M, \ M_1=MA_{1, s+2k+4\ell}^{-1}, \ M_2=M_1A_{1, s+2k+4\ell-2}^{-1}, \ \ldots, \nonumber \\ & M_{2\ell}=M_{2\ell-1}A_{1, s+2k+2}^{-1}, \ M_{2\ell+1}=M_{2\ell}A_{2, s+2k-1}^{-1}A_{1, s+2k}^{-1}, \nonumber \\ & M_{2\ell+2}=M_{2\ell+1}A_{2, s+2k-3}^{-1}, \ M_{2\ell+3}=M_{2\ell+2}A_{2, s+2k-5}^{-1}, \ \ldots, \ M_{k+2\ell}=M_{k+2\ell-1}A_{2, s+1}^{-1}.
\end{align*}
The dominant monomials in $\chi_q(R_{k, 2\ell+1, \ell}^{(s)} )  \chi_q( R_{k-1, 2\ell, \ell}^{(s+2)})$ are $M_0, \ldots, M_{k+2\ell-1}$.

\item Let $$M=R_{k, 2\ell+1, \ell}^{(s)} R_{k-1, 2\ell+2, \ell}^{(s+2)}, k\geq 1, \ell \geq 1. $$ Then dominant monomials in $\chi_q( R_{k, 2\ell+1, \ell}^{(s)} ) \chi_q( R_{k-1, 2\ell+2, \ell}^{(s+2)} )$ are
\begin{align*}
& M_0=M, \ M_1=MA_{1, s+2k+4\ell+2}^{-1}, \ M_2=M_1A_{1, s+2k+4\ell}^{-1}, \ \ldots, \nonumber \\ & M_{2\ell+1}=M_{2\ell}A_{1, s+2k+2}^{-1}, \ M_{2\ell+2}=M_{2\ell+1}A_{2, s+2k-1}^{-1}A_{1, s+2k}^{-1}, \nonumber \\ & M_{2\ell+3}=M_{2\ell+2}A_{2, s+2k-3}^{-1}, \ M_{2\ell+4}=M_{2\ell+3}A_{2, s+2k-5}^{-1}, \ \ldots, \ M_{k+2\ell+1}=M_{k+2\ell}A_{2, s+1}^{-1}.
\end{align*}
The dominant monomials in $\chi_q(R_{k, 2\ell+2, \ell}^{(s)} )  \chi_q( R_{k-1, 2\ell+1, \ell}^{(s+2)})$ are $M_0, \ldots, M_{k+2\ell}$.

\item Let $$M=\tilde{T}_{k, \ell-1, 0}^{(s)} \tilde{T}_{k-1, \ell}^{(s+2)}, k\geq 1, \ell \geq 1. $$ Then dominant monomials in $\chi_q(\tilde{T}_{k, \ell-1, 0}^{(s)} )  \chi_q( \tilde{T}_{k-1, \ell, 0}^{(s+2)})$ are
\begin{align*}
& M_0=M, \ M_1=MA_{2, s+2k+2\ell-2}^{-1}, \ M_2=M_1A_{2, s+2k+2\ell-4}^{-1}, \ \ldots, \nonumber \\ & M_{\ell-1}=M_{\ell-2}A_{2, s+2k+2}^{-1}, \ M_{\ell}=M_{\ell-1}A_{1, s+2k-1}^{-1}A_{2, s+2k}^{-1}, \nonumber \\ & M_{\ell+1}=M_{\ell}A_{1, s+2k-3}^{-1}, \ M_{\ell+2}=M_{\ell+1}A_{1, s+2k-5}^{-1}, \ \ldots, \ M_{k+\ell-1}=M_{k+\ell-2}A_{1, s+1}^{-1}.
\end{align*}
The dominant monomials in $\chi_q(\tilde{T}_{k, \ell, 0}^{(s)} )  \chi_q( \tilde{T}_{k-1, \ell-1, 0}^{(s+2)})$ are $M_0, \ldots, M_{k+\ell-2}$.

Recall that $\sigma(x)=0$ if $x$ is even and $\sigma(x)=1$ if $x$ is odd.
\item Let $$M=U_{p, \ell-1}^{(s)} U_{p-1, \ell}^{(s+2)}, p\geq 2, \ell \geq 1. $$ Then dominant monomials in $\chi_q( U_{p, \ell-1}^{(s)} )  \chi_q( U_{p-1, \ell}^{(s+2)} )$ are
\begin{align*}
& M_0=M, \ M_1=MA_{3, s+2p+2\ell-3}^{-1}, \ M_2=M_1A_{2, s+2p+2\ell-7}^{-1}, \ \ldots, \nonumber \\ & M_{ \lfloor \frac{\ell-1}{2} \rfloor }=M_{ \lfloor \frac{\ell-1}{2} \rfloor -1 }A_{3, s+2p+2\sigma(\ell+1)+3}^{-1}, \ M_{ \lfloor \frac{\ell-1}{2} \rfloor + 1 }=M_{ \lfloor \frac{\ell-1}{2} \rfloor }A_{2, s+2p-1}^{-1}A_{3, s+2p+1}^{-1}, \nonumber \\ & M_{ \lfloor \frac{\ell-1}{2} \rfloor + 2}=M_{ \lfloor \frac{\ell-1}{2} \rfloor + 1 }A_{2, s+2p-3}^{-1}, \ M_{\lfloor \frac{\ell-1}{2} \rfloor+3}=M_{\lfloor \frac{\ell-1}{2} \rfloor+2}A_{2, s+2p-5}^{-1}, \ \ldots, \\
& M_{p+\lfloor \frac{\ell-1}{2} \rfloor }=M_{p+\lfloor \frac{\ell-1}{2} \rfloor-1}A_{2, s+1}^{-1}.
\end{align*}
The dominant monomials in $\chi_q( U_{p, \ell}^{(s)} )  \chi_q( U_{p-1, \ell-1}^{(s+2)} )$ are $M_0, \ldots, M_{p+\lfloor \frac{\ell-1}{2} \rfloor-1}$.

\item Let $$M=V_{k, \ell-1}^{(s)} V_{k-1, \ell}^{(s+4)}, k\geq 1, \ell \geq 1. $$ Then dominant monomials in $\chi_q( V_{k, \ell-1}^{(s)} )  \chi_q( V_{k-1, \ell}^{(s+4)} )$ are
\begin{align*}
& M_0=M, \ M_1=MA_{3, s+4k+4\ell-4}^{-1}, \ M_2=M_1A_{3, s+4k+4\ell-8}^{-1}, \ \ldots, \nonumber \\ & M_{\ell-1}=M_{\ell-2}A_{3, s+4k+6}^{-1}, \ M_{\ell}=M_{\ell-1}A_{3, s+4k-2}^{-1}A_{2, s+4k}^{-1} A_{2, s+4k-2}^{-1}A_{3, s+4k}^{-1}, \nonumber \\ & M_{\ell+1}=M_{\ell}A_{3, s+4k-6}^{-1}, \ M_{\ell+2}=M_{\ell+1}A_{3, s+4k-10}^{-1}, \ \ldots, \ M_{k+\ell-1}=M_{k+\ell-2}A_{3, s+2}^{-1}.
\end{align*}
The dominant monomials in $\chi_q( V_{k, \ell}^{(s)} )  \chi_q( V_{k-1, \ell-1}^{(s+4)} )$ are $M_0, \ldots, M_{k+\ell-2}$.

\item Let $$M=P_{k, \ell-1}^{(s)} P_{k-1, \ell}^{(s+4)}, k\geq 1, \ell \geq 1. $$ Then dominant monomials in $\chi_q( P_{k, \ell-1}^{(s)} )  \chi_q( P_{k-1, \ell}^{(s+4)} )$ are
\begin{align*}
& M_0=M, \ M_1=MA_{3, s+4k+4\ell}^{-1}, \ M_2=M_1A_{3, s+4k+4\ell-4}^{-1}, \ \ldots, \nonumber \\
& M_{\ell-1}=M_{\ell-2}A_{3, s+4k+8}^{-1}, \ M_{\ell}=M_{\ell-1}A_{2, s+4k+2}^{-1}A_{3, s+4k+4}^{-1}, \nonumber \\ & M_{\ell+1}=M_{\ell}A_{3, s+4k-2}^{-1}A_{2, s+4k}^{-1}, \ M_{\ell+2}=M_{\ell+1}A_{3, s+4k-6}^{-1}, \\
& M_{\ell+3}=M_{\ell+2}A_{3, s+4k-10}^{-1}, \ \ldots, \ M_{k+\ell}=M_{k+\ell-1}A_{3, s+2}^{-1}.
\end{align*}
The dominant monomials in $\chi_q( P_{k, \ell}^{(s)} )  \chi_q( P_{k-1, \ell-1}^{(s+4)} )$ are $M_0, \ldots, M_{k+\ell-1}$.

\item Let $$M=O_{k, \ell-1}^{(s)} O_{k-1, \ell}^{(s+2)}, k\geq 1, \ell \geq 1. $$ Then dominant monomials in $\chi_q( O_{k, \ell-1}^{(s)} )  \chi_q( O_{k-1, \ell}^{(s+2)} )$ are
\begin{align*}
& M_0=M, \ M_1=MA_{2, s+2k+2\ell+4}^{-1}, \ M_2=M_1A_{2, s+2k+2\ell+2}^{-1}, \ \ldots, \nonumber \\
& M_{\ell-1}=M_{\ell-2}A_{2, s+2k+7}^{-1}, \ M_{\ell}=M_{\ell-1}A_{1, s+2k+4}^{-1}A_{2, s+2k+5}^{-1}, \nonumber \\
& M_{\ell+1}=M_{\ell}A_{1, s+2k+2}^{-1}, \ M_{\ell+2}=M_{\ell+1}A_{2, s+2k-1}^{-1}A_{1, 2k}^{-1}, \\
& M_{\ell+3}=M_{\ell+2}A_{2, s+2k-3}^{-1}, \ M_{\ell+4}=M_{\ell+3}A_{2, s+2k-5}^{-1}, \ \ldots, \ M_{k+\ell+1}=M_{k+\ell}A_{2, s+1}^{-1}.
\end{align*}
The dominant monomials in $\chi_q( O_{k, \ell}^{(s)} )  \chi_q( O_{k-1, \ell-1}^{(s+2)} )$ are $M_0, \ldots, M_{k+\ell}$.

\end{enumerate}
In each case, for each $i$, the multiplicity of $M_i$ in the corresponding product of $q$-characters is $1$.
\end{lemma}

\begin{proof}

We prove the case of $\chi_q(T_{k, \ell-1, 0}^{(s)} )  \chi_q( T_{k-1, \ell, 0}^{(s+4)} )$. The other cases are similar.
Let $m'_1=T_{k, \ell-1, 0}^{(s)}$, $m'_2=T_{k-1, \ell, 0}^{(s+4)}$. Without loss of generality, we assume that $s=0$. Then
\begin{align*}
& m'_1=(3_{0}3_{4} \cdots 3_{4k-4})(2_{4k+1}2_{4k+3}\cdots 2_{4k+2\ell - 3}), \\
& m'_2=(3_{4}3_{6} \cdots 3_{4k-4})(2_{4k+1}2_{4k+3}\cdots 2_{4k+2\ell - 1}).
\end{align*}

Let $m=m_1m_2$ be a dominant monomial, where $m_i \in \chi_q(m'_i), i=1, 2$.
Let
\begin{align*}
m_3=3_{4}3_{6} \cdots 3_{4k-4}, \quad m_4=2_{4k+1}2_{4k+3}\cdots 2_{4k+2\ell - 1}.
\end{align*}
If $m_2 \in \chi_q(m_3)(\chi_q(m_4)-m_4)$, then $m=m_1m_2$ is right negative and hence $m$ is not dominant. Therefore $m_2 \in \chi_q(m_3)m_4$.

Suppose that $m_2 \in \mathscr{M}(L(m'_2)) \cap \mathscr{M}((\chi_q(m_3)-m_3)m_4)$.  By the FM algorithm for $L(m'_2)$, $m_2$ must have a factor $3_{4k}^{-1}$. By Lemma \ref{fundamental q-characters}, $m_1$ does not have the factor $3_{4k}$. Therefore $m=m_1 m_2$ is not dominant. This is a contradiction. Therefore $m_2=m_3m_4=m'_2$.

If
\begin{align*}
m_1 \in \chi_q( & 3_{0}3_{4} \cdots 3_{4k-4} 2_{4k+1}2_{4k+3}\cdots 2_{4k+2\ell - 5} ) \times \\
& \times ( \chi_q( 2_{4k+2\ell - 3} ) - 2_{4k+2\ell - 3} - 2_{4k+2\ell - 3}^{-1} 1_{4k+2\ell - 4} 3_{4k+2\ell - 4} ),
\end{align*}
then $m=m_1m_2$ is right-negative and hence not dominant. Therefore $m_1$ is in one of the following polynomials
\begin{align}
& \chi_q( 3_{0}3_{4} \cdots 3_{4k-4} 2_{4k+1}2_{4k+3}\cdots 2_{4k+2\ell - 5} ) 2_{4k+2\ell - 3}, \label{first set} \\
& \chi_q( 3_{0}3_{4} \cdots 3_{4k-4} 2_{4k+1}2_{4k+3}\cdots 2_{4k+2\ell - 5} ) 2_{4k+2\ell - 1}^{-1} 1_{4k+2\ell - 2} 3_{4k+2\ell - 2}. \label{second set}
\end{align}

If $m_1$ is in (\ref{first set}), then $m_1=m'_1$. The dominant monomial we obtain is $M_0=m'_1m'_2$. If $m_1$ is the highest monomial in (\ref{second set}), then we obtain the dominant monomial $M_1=m_1m'_2$. Suppose that $m_1$ is in
\begin{eqnarray*}
\mathscr{M}(L(m'_1)) & \cap & \mathscr{M}\big( \chi_q(3_{0}3_{4} \cdots 3_{4k-4} 2_{4k+1}2_{4k+3}\cdots 2_{4k+2\ell - 7}) \times \\
&& \qquad \times (\chi_q(2_{4k+2\ell - 5}) - 2_{4k+2\ell - 5} ) 2_{4k+2\ell - 3}^{-1} 1_{4k+2\ell - 4} 3_{4k+2\ell - 4} \big).
\end{eqnarray*}
By the FM algorithm for $L(m'_1)$, $m_1$ is in one of the following polynomials
\begin{align}
& \chi_q(3_{0}3_{4} \cdots 3_{4k-4} 2_{4k+1}2_{4k+3}\cdots 2_{4k+2\ell - 7}) 2_{4k+2\ell - 3}^{-1} 1_{4k+2\ell - 4} 3_{4k+2\ell - 4} 2_{4k+2\ell - 1}^{-1} 1_{4k+2\ell - 2} 3_{4k+2\ell - 2}, \label{set 1} \\
& \chi_q(3_{0}3_{4} \cdots 3_{4k-4} 2_{4k+1}2_{4k+3}\cdots 2_{4k+2\ell - 7}) 1_{4k+2\ell - 4} 3_{4k+2\ell}^{-1} 1_{4k+2\ell - 2} 3_{4k+2\ell - 2}.      \label{set 2}
\end{align}
If $m_1$ is in (\ref{set 2}), then $m=m_1m_2$ is right-negative and hence $m$ is not dominant. This is a contradiction. Therefore $m_1$ is in (\ref{set 1}). If $m_1$ is the highest monomial in (\ref{set 1}), then we obtain the dominant monomial $M_2=m_1m'_2$. Suppose that $m_1$ is not the highest monomial in (\ref{set 1}). Then by the FM algorithm, $m_1$ is in the set
\begin{align*}
& \chi_q(3_{0}3_{4} \cdots 3_{4k-4} 2_{4k+1}2_{4k+3}\cdots 2_{4k+2\ell - 9}) \times \\
& \times 2_{4k+2\ell - 5}^{-1} 1_{4k+2\ell - 6} 3_{4k+2\ell - 6} 2_{4k+2\ell - 3}^{-1} 1_{4k+2\ell - 4} 3_{4k+2\ell - 4} 2_{4k+2\ell - 1}^{-1} 1_{4k+2\ell - 2} 3_{4k+2\ell - 2}.
\end{align*}
If $m_1$ is the highest monomial in the above set, then we obtain the dominant monomial $M_3=m_1m'_2$. Continue this procedure, we obtain dominant monomials $M_4, \ldots, M_{\ell-1}$ and the remaining dominant monomials are of the form $m_1m'_2$, where $m_1$ is a non-highest monomial in
\begin{align*}
\mathscr{M}(L(m'_1)) \cap \mathscr{M}\big( & \chi_q(3_{0}3_{4} \cdots 3_{4k-4}) \times \\
& \times 2_{4k+3}^{-1}2_{4k+5}^{-1}\cdots 2_{4k+2\ell-1}^{-1} 1_{4k+2}1_{4k+4} \cdots 1_{4k+2\ell-2} 3_{4k+2}3_{4k+4} \cdots 3_{4k+2\ell-2} \big).
\end{align*}
Suppose that $m_1$ is a non-highest monomial in the above set. Since the non-highest monomials in $\chi_q(3_{0}3_{4} \cdots 3_{4k-4})$ are right-negative, we need cancelations of factors with negative powers of some monomial in $\chi_q(3_{0}3_{4} \cdots 3_{4k-4})$ with
$$
1_{4k+2}1_{4k+4} \cdots 1_{4k+2\ell-2} 3_{4k+2}3_{4k+4} \cdots 3_{4k+2\ell-2} 2_{4k+1}.
$$
The only cancelation can happen is to cancel $1_{4k+2}$ or $3_{4k+2}$ or $3_{4k+4}$ or $2_{4k+1}$. Since $1_{4k}^{2}$, $3_{4k}^2$, and $3_{4k-2}^2$ do not appear in
\begin{align*}
\mathscr{M}(L(m'_1)) \cap \mathscr{M}\big( & \chi_q(3_{0}3_{4} \cdots 3_{4k-4}) \times \\
& \times 2_{4k+3}^{-1}2_{4k+5}^{-1}\cdots 2_{4k+2\ell-1}^{-1} 1_{4k+2}1_{4k+4} \cdots 1_{4k+2\ell-2} 3_{4k+2}3_{4k+4} \cdots 3_{4k+2\ell-2} \big),
\end{align*}
$1_{4k+2}$, $3_{4k+2}$ and $3_{4k+4}$ cannot be canceled. Therefore we need a cancelation with $2_{4k+1}$. The only monomials in $\chi_q(3_{0}3_{4} \cdots 3_{4k-4})$ which can cancel $2_{4k+1}$ is in one of the following polynomials
\begin{align*}
& \chi_q(3_{0}3_{4} \cdots 3_{4k-8}) 1_{4k} 2_{4k-3} 2_{4k+1}^{-1} , \\
& \chi_q(3_{0}3_{4} \cdots 3_{4k-8}) 1_{4k-2} 1_{4k} 2_{4k-1}^{-1} 2_{4k+1}^{-1} 3_{4k-2} , \\
& \chi_q(3_{0}3_{4} \cdots 3_{4k-8}) 2_{4k+1}^{-1} 2_{4k+3}^{-1} 3_{4k}.
\end{align*}
Therefore $m_1$ is in one of the following sets
\begin{align}
\mathscr{M}(L(m'_1)) \cap \mathscr{M}\big( & \chi_q(3_{0}3_{4} \cdots 3_{4k-8}) 1_{4k} 2_{4k-3} 2_{4k+1}^{-1} \times \nonumber \\
& \times 2_{4k+3}^{-1}2_{4k+5}^{-1}\cdots 2_{4k+2\ell-1}^{-1} 1_{4k+2}1_{4k+4} \cdots 1_{4k+2\ell-2} 3_{4k+2}3_{4k+4} \cdots 3_{4k+2\ell-2} \big), \label{one} \\
\mathscr{M}(L(m'_1)) \cap \mathscr{M}\big( & \chi_q(3_{0}3_{4} \cdots 3_{4k-8}) 1_{4k-2} 1_{4k} 2_{4k-1}^{-1} 2_{4k+1}^{-1} 3_{4k-2} \times \nonumber \\
& \times 2_{4k+3}^{-1}2_{4k+5}^{-1}\cdots 2_{4k+2\ell-1}^{-1} 1_{4k+2}1_{4k+4} \cdots 1_{4k+2\ell-2} 3_{4k+2}3_{4k+4} \cdots 3_{4k+2\ell-2} \big), \label{two} \\
\mathscr{M}(L(m'_1)) \cap \mathscr{M}\big( & \chi_q(3_{0}3_{4} \cdots 3_{4k-8}) 2_{4k+1}^{-1} 2_{4k+3}^{-1} 3_{4k} \times \nonumber \\
& \times 2_{4k+3}^{-1}2_{4k+5}^{-1}\cdots 2_{4k+2\ell-1}^{-1} 1_{4k+2}1_{4k+4} \cdots 1_{4k+2\ell-2} 3_{4k+2}3_{4k+4} \cdots 3_{4k+2\ell-2} \big). \label{three}
\end{align}

If $m_1$ is in (\ref{two}), then we need to cancel $2_{4k-1}^{-1}$. But $2_{4k-1}^{-1}$ cannot be canceled by any monomials in $\chi_q(3_{0}3_{4} \cdots 3_{4k-8})$ or by $m'_2$. Hence $m_1$ is not in (\ref{two}).

If $m_1$ is in (\ref{three}), then we need to cancel $2_{4k+3}^{-1}$. But $2_{4k+3}^{-1}$ cannot be canceled by any monomials in $\chi_q(3_{0}3_{4} \cdots 3_{4k-8})$ or by $m'_2$. Therefore $m_1$ is not in (\ref{three}). Hence $m_1$ is in (\ref{one}).

If $m_1$ is the highest monomial in (\ref{one}) with respect to $\leq$ defined in (\ref{partial order of monomials}), then we obtain the dominant monomial $M_{\ell} = m_1m'_2$. Suppose that $m_1$ a non-highest monomial in (\ref{one}). By the FM algorithm, $m_1$ must in
\begin{align*}
\mathscr{M}(L(m'_1)) \cap \mathscr{M}\big( & \chi_q(3_{0}3_{4} \cdots 3_{4k-12}) 3_{4k-4}^{-1} 2_{4k-7} 2_{4k-5} 1_{4k} 2_{4k-3} 2_{4k+1}^{-1} \times \nonumber \\
& \times 2_{4k+3}^{-1}2_{4k+5}^{-1}\cdots 2_{4k+2\ell-1}^{-1} 1_{4k+2}1_{4k+4} \cdots 1_{4k+2\ell-2} 3_{4k+2}3_{4k+4} \cdots 3_{4k+2\ell-2} \big).
\end{align*}
If $m_1$ is the highest monomial in the above set, then we obtain the dominant monomial $M_{\ell+1} = m_1m'_2$. Continue this procedure, we can show that the only remaining dominant monomials are $M_{\ell+2}, \ldots, M_{k+\ell-1}$.

It is clear that the multiplicity of $M_i, i=1, \ldots, k+\ell-1$, in $\chi_q(m_1)  \chi_q(m_2)$ is $1$.
\end{proof}

\subsection{Products of sources are special}
\begin{lemma} \label{dominant monomials in sources}
Let $[\mathcal{S}]$ be the last summand in one of the relations in Theorem \ref{the system I}. Then $\mathcal{S}$ is special.
\end{lemma}

\begin{proof}
We prove the case of $\mathcal{S}$ in the first relation in Theorem \ref{the system I}. The other cases are similar.

Let $S=\chi_q( R_{2k-1, \ell,  \lfloor \frac{\ell}{2} \rfloor}^{(s+1)} )  \chi_q( T_{\lfloor \frac{\ell-1}{2} \rfloor, 0, 0}^{(s+4k+4)} ) $, $\ell=2m$, $m\in \mathbb{Z}_{\geq 0}$. Let
\begin{align*}
& n_1=2_{s+1}2_{s+3}\cdots 2_{s+4k-3}, \ n'_1=1_{s+4k}1_{s+4k+2}\cdots 1_{s+4k+4m-2}, \\
& n''_1=3_{s+4k+2} 3_{s+4k+6} \cdots 3_{s+4k+4m-2}, \ n_2=3_{s+4k+4}3_{s+4k+8}\cdots 3_{s+4k+4m-4}.
\end{align*}
Then $R_{2k-1, \ell,  \lfloor \frac{\ell}{2} \rfloor}^{(s+1)}=n_1n'_1n''_1, T_{\lfloor \frac{\ell-1}{2} \rfloor, 0, 0}^{(s+4k+4)}=n_2$.
Let $m=m_1m_2$ be a dominant monomial, where
\begin{align*}
m_1\in \chi_q(R_{2k-1, \ell,  \lfloor \frac{\ell}{2} \rfloor}^{(s+1)}), \ m_2\in \chi_q(T_{\lfloor \frac{\ell-1}{2} \rfloor, 0, 0}^{(s+4k+4)}).
\end{align*}

If $m_2 \neq n_2$ or $m_1 \in \chi_q(n_1 n_1'')(\chi_q(n'_1) - n'_1 )$ or $m_1 \in \chi_q(n_1n_1')(\chi_q(n''_1) - n''_1 )$, then $m$ is right-negative which contradicts the fact that $m$ is dominant. Therefore $m_2= n_2$, $m_1 \in \chi_q(n_1)n'_1n_1''$.

If $m_1$ is the highest monomial in $\chi_q(n_1)n'_1n_1''$, then we obtain a dominant monomial $n_1n'_1n_1''n_2$. Suppose that $m_1$ is a non-highest monomial in $\chi_q(n_1)n'_1n_1''$. Then $m_1$ is in $(\chi_q(n_1)-n_1)n'_1n_1''$. By the FM algorithm, $m_1$ is one of the following monomials
\begin{align*}
\bar{m}_1=n_1n_1'n_1''A_{2, s+4k-2}^{-1}, \ \bar{m}_2=\bar{m}_1A_{2, s+4k-4}^{-1}, \ \ldots, \ \bar{m}_{2k-1}=\bar{m}_{2k-2}A_{2, s+2}^{-1}.
\end{align*}
Therefore $m=m_1m_2=m_1n_2$ is one of the following monomials
\begin{align*}
\bar{m}_1n_2=n_1n_1'n_1''A_{2, s+4k-2}^{-1}n_2, \ \bar{m}_2n_2=\bar{m}_1A_{2, s+4k-4}^{-1}n_2, \ \ldots, \ \bar{m}_{2k-1}n_2=\bar{m}_{2k-2}A_{2, s+2}^{-1}n_2.
\end{align*}
Hence $m$ is not a dominant monomial. This is a contradiction. Therefore the only dominant monomial in $S$ is $R_{2k-1, \ell,  \lfloor \frac{\ell}{2} \rfloor}^{(s+1)} T_{\lfloor \frac{\ell-1}{2} \rfloor, 0, 0}^{(s+4k+4)} $.
\end{proof}

\subsection{Proof of Theorem \ref{the system I}}
In Section \ref{prove special}, we have shown that $\mathcal{R}_{0, 2\ell+i, \ell}^{(s)}$, $i=0, 1, 2$, and $\mathcal{U}_{k, \ell}^{(s)}$ are special. Therefore by Lemma \ref{dominant monomials in sources}, (\ref{relation 9 in the system I}) and (\ref{relation 2 in system 3}) are true. By Lemmas \ref{dominant monomials} and \ref{dominant monomials in sources}, the dominant monomials in the $q$-characters of the left hand side and of the right hand side of every relation in Theorem \ref{the system I} are the same. The theorem follows.

\section{Proof of Theorem \ref{irreducible} and Theorem \ref{irreducible 3}} \label{prove irreducible}

By Lemma \ref{dominant monomials in sources}, $\mathcal{S}$ is special and hence irreducible. Therefore we only have to show that each $\mathcal{T}\otimes \mathcal{B}$ in Theorem \ref{irreducible} and Theorem \ref{irreducible 3} is irreducible. It suffices to prove that for each non-highest dominant monomial $M$ in $\mathcal{T} \otimes \mathcal{B}$, we have $\mathscr{M}(L(M)) \not\subset \mathscr{M}(\mathcal{T} \otimes \mathcal{B})$. The idea is similar as in \cite{Her06}, \cite{MY12b}, \cite{LM12}.

\begin{lemma} \label{monomials that are not in tensor product}
We consider the same cases as in Lemma \ref{dominant monomials}. In each case $M_i$ are the dominant monomials described by that lemma.
\begin{enumerate}[(1)]
\item For $k\geq 1, \ell \geq 1$, let
\begin{align*}
& n_1=M_1A_{2, s+4k+2\ell-2}^{-1}, \ n_2=M_2A_{2, s+4k+2\ell-4}^{-1}, \ \ldots, \nonumber \\ & n_{\ell-1}=M_{\ell-1}A_{2, s+4k+2}^{-1}, \ n_{\ell}=M_{\ell}A_{3, s+4k-2}^{-1}A_{2, s+4k}^{-1}, \nonumber \\ & n_{\ell+1}=M_{\ell+1}A_{3, s+4k-6}^{-1}, \ n_{\ell+2}=M_{\ell+2}A_{3, s+4k-10}^{-1}, \ \ldots, \ n_{k+\ell-2}=M_{k+\ell-2}A_{3, s+6}^{-1}.
\end{align*}
Then for $i=1, \ldots, k+\ell-2$, $n_i\in \chi_q(M_i)$ and $n_i \not\in \chi_q(T_{k, \ell, 0}^{(s)} )  \chi_q( T_{k-1, \ell-1, 0}^{(s+4)})$.

\item For $k\geq 1, m \geq 1$, let
\begin{align*}
& n_1=M_1A_{1, s+4k+2m-1}^{-1}, \ n_2=M_2A_{1, s+4k+2m-3}^{-1}, \ \ldots, \nonumber \\ & n_{m-1}=M_{m-1}A_{1, s+4k+3}^{-1}, \ n_{m}=M_{m}A_{3, s+4k-2}^{-1}A_{2, s+4k}^{-1}A_{1, s+4k+1}^{-1}, \nonumber \\ & n_{m+1}=M_{m+1}A_{3, s+4k-6}^{-1}, \ n_{m+2}=M_{m+2}A_{3, s+4k-10}^{-1}, \ \ldots, \ n_{k+m-2}=M_{k+m-2}A_{3, s+6}^{-1}.
\end{align*}
Then for $i=1, \ldots, k+m-2$, $n_i\in \chi_q(M_i)$ and $n_i \not\in \chi_q(T_{k, 0, m}^{(s)} )  \chi_q( T_{k-1, 0, m-1}^{(s+4)})$.

\item For $k\geq 1, \ell \geq 1$, let
\begin{align*}
& n_1=M_1A_{2, s+2k+2\ell+1}^{-1}, \ n_2=M_2A_{2, s+2k+2\ell-1}^{-1}, \ \ldots,  \nonumber \\ &  n_{\ell-1}=M_{\ell-1}A_{2, s+2k+5}^{-1}, \ n_{\ell}=M_{\ell}A_{2, s+2k-1}^{-1}A_{3, s+2k+1}^{-1}A_{2, 2k+3}^{-1}, \ \nonumber \\ &  n_{\ell+1}=M_{\ell+1}A_{2, s+2k-3}^{-1}, \ n_{\ell+2}=M_{\ell+2}A_{2, s+2k-5}^{-1}, \ \ldots, n_{k+\ell-2}=M_{k+\ell-2}A_{2, s+3}^{-1}.
\end{align*}
Then for $i=1, \ldots, k+\ell-2$, $n_i\in \chi_q(M_i)$ and $n_i \not\in \chi_q( S_{k, \ell}^{(s)} ) \chi_q( S_{k-1, \ell-1}^{(s+2)} )$.

\item For $\ell \geq 1, r \in \{1, 2\}$, let
\begin{align*}
& n_1=M_1A_{1, s+2\ell+2r-2}^{-1}, \ n_2=M_2 A_{1, s+2\ell+2r-4}^{-1}, \ \ldots, \nonumber \\ & n_{r-1}=M_{r-1}A_{1, s+2\ell+2}^{-1}, \ n_{r}=M_{r}A_{2, s+2\ell}^{-1}A_{1, s+2\ell+1}^{-1}, \nonumber \\ & n_{r+1}=M_{r+1}A_{2, s+2\ell-2}^{-1}, \ n_{r+2}=M_{r+2}A_{2, s+2\ell-4}^{-1}, \ \ldots, \ n_{\ell+r-2}=M_{\ell+r-2}A_{2, s+4}^{-1}.
\end{align*}
Then for $i=1, \ldots, \ell+r-2$, $n_i\in \chi_q(M_i)$ and $n_i \not\in \chi_q(T_{0, \ell, r}^{(s)} )  \chi_q( T_{0, \ell-1, r-1}^{(s+2)})$.

\item For $k, m \geq 1$, let
\begin{align*}
& n_1=M_1A_{3, s+2k+4m-2}^{-1}, \ n_2=M_2A_{3, s+2k+4m-6}^{-1}, \ \ldots, \nonumber \\ & n_{m-1}=M_{m-1}A_{3, s+2k+6}^{-1}, \ n_{m}=M_{m}A_{1, s+2k-1}^{-1}A_{2, s+2k}^{-1}A_{3, 2k+2}^{-1}, \nonumber \\ & n_{m+1}=M_{m+1}A_{1, s+2k-3}^{-1}, \ n_{m+2}=M_{m+2}A_{1, s+2k-5}^{-1}, \ \ldots, \ n_{k+m-2}=M_{k+m-2}A_{1, s+3}^{-1}.
\end{align*}
Then for $i=1, \ldots, k+m-2$, $n_i\in \chi_q(M_i)$ and $n_i \not\in \chi_q(\tilde{T}_{k, 0, m}^{(s)} )  \chi_q( \tilde{T}_{k-1, 0, m-1}^{(s+2)})$.

\item For $k\geq 1, \ell \geq 1$, let
\begin{align*}
& n_1=M_1A_{3, s+2k+4\ell-3}^{-1}, \ n_2=M_2 A_{3, s+2k+4\ell-7}^{-1}, \ \ldots, \nonumber \\ & n_{\ell-1}=M_{\ell-1}A_{3, s+2k+5}^{-1}, \ n_{\ell}=M_{\ell}A_{2, s+2k-1}^{-1}A_{3, s+2k+1}^{-1}, \nonumber \\ & n_{\ell+1}=M_{\ell+1}A_{2, s+2k-3}^{-1}, \ n_{\ell+2}=M_{\ell+2}A_{2, s+2k-5}^{-1}, \ \ldots, \ n_{k+\ell-2}=M_{k+\ell-2}A_{2, s+3}^{-1}.
\end{align*}
Then for $i=1, \ldots, k+\ell-2$, $n_i\in \chi_q(M_i)$ and $n_i \not\in \chi_q(R_{k, 2\ell, \ell}^{(s)} )  \chi_q( R_{k-1, 2\ell, \ell-1}^{(s+2)})$.

\item For $k\geq 1, \ell \geq 1$, let
\begin{align*}
& n_1=M_1A_{1, s+2k+4\ell}^{-1}, \ n_2=M_2A_{1, s+2k+4\ell-2}^{-1}, \ \ldots, \nonumber \\ & n_{2\ell}=M_{2\ell}A_{1, s+2k+2}^{-1}, \ n_{2\ell+1}=M_{2\ell+1}A_{2, s+2k-1}^{-1}A_{1, s+2k}^{-1}, \nonumber \\ & n_{2\ell+2}=M_{2\ell+2}A_{2, s+2k-3}^{-1}, \ n_{2\ell+3}=M_{2\ell+3}A_{2, s+2k-5}^{-1}, \ \ldots, \ n_{k+2\ell-1}=M_{k+2\ell-1}A_{2, s+3}^{-1}.
\end{align*}
Then for $i=1, \ldots, k+2\ell-1$, $n_i\in \chi_q(M_i)$ and $n_i \not\in \chi_q(R_{k, 2\ell+1, \ell}^{(s)} )  \chi_q( R_{k-1, 2\ell, \ell}^{(s+2)})$.

\item For $k\geq 1, \ell \geq 1$, let
\begin{align*}
& n_1=M_1A_{1, s+2k+4\ell+2}^{-1}, \ n_2=M_2A_{1, s+2k+4\ell}^{-1}, \ \ldots, \nonumber \\ & n_{2\ell+1}=M_{2\ell+1}A_{1, s+2k+2}^{-1}, \ n_{2\ell+2}=M_{2\ell+2}A_{2, s+2k-1}^{-1}A_{1, s+2k}^{-1}, \nonumber \\ & n_{2\ell+3}=M_{2\ell+3}A_{2, s+2k-3}^{-1}, \ n_{2\ell+4}=M_{2\ell+4}A_{2, s+2k-5}^{-1}, \ \ldots, \ n_{k+2\ell}=M_{k+2\ell}A_{2, s+3}^{-1}.
\end{align*}
Then for $i=1, \ldots, k+2\ell$, $n_i\in \chi_q(M_i)$ and $n_i \not\in \chi_q(R_{k, 2\ell+2, \ell}^{(s)} )  \chi_q( R_{k-1, 2\ell+1, \ell}^{(s+2)})$.

\item For $k\geq 1, \ell \geq 1$, let
\begin{align*}
& n_1=M_1A_{2, s+2k+2\ell-2}^{-1}, \ n_2=M_2A_{2, s+2k+2\ell-4}^{-1}, \ \ldots, \nonumber \\ & n_{\ell-1}=M_{\ell-1}A_{2, s+2k+2}^{-1}, \ n_{\ell}=M_{\ell}A_{1, s+2k-1}^{-1}A_{2, s+2k}^{-1}, \nonumber \\ & n_{\ell+1}=M_{\ell+1}A_{1, s+2k-3}^{-1}, \ n_{\ell+2}=M_{\ell+2}A_{1, s+2k-5}^{-1}, \ \ldots, \ n_{k+\ell-2}=M_{k+\ell-2}A_{1, s+3}^{-1}.
\end{align*}
Then for $i=1, \ldots, k+\ell-2$, $n_i\in \chi_q(M_i)$ and $n_i \not\in \chi_q(\tilde{T}_{k, \ell, 0}^{(s)} )  \chi_q( \tilde{T}_{k-1, \ell-1, 0}^{(s+2)})$.

\item For $p\geq 2, \ell \geq 1$, let
\begin{align*}
& n_1=M_1A_{3, s+2p+2\ell-3}^{-1}, \ n_2=M_2A_{2, s+2p+2\ell-7}^{-1}, \ \ldots, \nonumber \\ & n_{ \lfloor \frac{\ell-1}{2} \rfloor }=M_{ \lfloor \frac{\ell-1}{2} \rfloor }A_{3, s+2p+2\sigma(\ell+1)+3}^{-1}, \ n_{ \lfloor \frac{\ell-1}{2} \rfloor + 1 }=M_{ \lfloor \frac{\ell-1}{2} \rfloor + 1 }A_{2, s+2p-1}^{-1}A_{3, s+2p+1}^{-1}, \nonumber \\ & n_{ \lfloor \frac{\ell-1}{2} \rfloor + 2}=M_{ \lfloor \frac{\ell-1}{2} \rfloor + 2 }A_{2, s+2p-3}^{-1}, \ n_{\lfloor \frac{\ell-1}{2} \rfloor+3}=M_{\lfloor \frac{\ell-1}{2} \rfloor+3}A_{2, s+2p-5}^{-1}, \ \ldots, \\
& n_{p+\lfloor \frac{\ell-1}{2} \rfloor - 1 }=M_{p+\lfloor \frac{\ell-1}{2} \rfloor - 1}A_{2, s+3}^{-1}.
\end{align*}
Then for $i=1, \ldots, p+\lfloor \frac{\ell-1}{2} \rfloor - 1$, $n_i\in \chi_q(M_i)$ and $n_i \not\in \chi_q( U_{p, \ell}^{(s)} )  \chi_q( U_{p-1, \ell-1}^{(s+2)} )$.

\item For $k\geq 1, \ell \geq 1$, let
\begin{align*}
& n_1=M_1A_{3, s+4k+4\ell-4}^{-1}, \ n_2=M_2A_{3, s+4k+4\ell-8}^{-1}, \ \ldots, \nonumber \\ & n_{\ell-1}=M_{\ell-1}A_{3, s+4k+6}^{-1}, \ n_{\ell}=M_{\ell}A_{3, s+4k-2}^{-1}A_{2, s+4k}^{-1} A_{2, s+4k-2}^{-1}A_{3, s+4k}^{-1}, \nonumber \\ & n_{\ell+1}=M_{\ell+1}A_{3, s+4k-6}^{-1}, \ n_{\ell+2}=M_{\ell+2}A_{3, s+4k-10}^{-1}, \ \ldots, \ n_{k+\ell-2}=M_{k+\ell-2}A_{3, s+6}^{-1}.
\end{align*}
Then for $i=1, \ldots, k+\ell-2$, $n_i\in \chi_q(M_i)$ and $n_i \not\in \chi_q( V_{k, \ell}^{(s)} )  \chi_q( V_{k-1, \ell-1}^{(s+4)} )$.

\item For $k\geq 1, \ell \geq 1$, let
\begin{align*}
& n_1=M_1A_{3, s+4k+4\ell}^{-1}, \ n_2=M_2A_{3, s+4k+4\ell-4}^{-1}, \ \ldots, \nonumber \\
& n_{\ell-1}=M_{\ell-1}A_{3, s+4k+8}^{-1}, \ n_{\ell}=M_{\ell}A_{2, s+4k+2}^{-1}A_{3, s+4k+4}^{-1}, \nonumber \\ & n_{\ell+1}=M_{\ell+1}A_{3, s+4k-2}^{-1}A_{2, s+4k}^{-1}, \ n_{\ell+2}=M_{\ell+2}A_{3, s+4k-6}^{-1}, \\
& n_{\ell+3}=M_{\ell+3}A_{3, s+4k-10}^{-1}, \ \ldots, \ n_{k+\ell-1}=M_{k+\ell-1}A_{3, s+6}^{-1}.
\end{align*}
Then for $i=1, \ldots, k+\ell-1$, $n_i\in \chi_q(M_i)$ and $n_i \not\in \chi_q( P_{k, \ell}^{(s)} )  \chi_q( P_{k-1, \ell-1}^{(s+4)} )$.

\item For $k\geq 1, \ell \geq 1$, let
\begin{align*}
& n_1=M_1A_{2, s+2k+2\ell+4}^{-1}, \ n_2=M_2A_{2, s+2k+2\ell+2}^{-1}, \ \ldots, \nonumber \\
& n_{\ell-1}=M_{\ell-1}A_{2, s+2k+7}^{-1}, \ n_{\ell}=M_{\ell}A_{1, s+2k+4}^{-1}A_{2, s+2k+5}^{-1}, \nonumber \\
& n_{\ell+1}=M_{\ell+1}A_{1, s+2k+2}^{-1}, \ n_{\ell+2}=M_{\ell+2}A_{2, s+2k-1}^{-1}A_{1, 2k}^{-1}, \\
& n_{\ell+3}=M_{\ell+3}A_{2, s+2k-3}^{-1}, \ n_{\ell+4}=M_{\ell+4}A_{2, s+2k-5}^{-1}, \ \ldots, \ n_{k+\ell}=M_{k+\ell}A_{2, s+3}^{-1}.
\end{align*}
Then for $i=1, \ldots, k+\ell$, $n_i\in \chi_q(M_i)$ and $n_i \not\in \chi_q( O_{k, \ell}^{(s)} )  \chi_q( O_{k-1, \ell-1}^{(s+2)} )$.

\end{enumerate}
\end{lemma}

\begin{proof}
We prove the case of $\chi_q(T_{k, \ell, 0}^{(s)} )  \chi_q( T_{k-1, \ell-1, 0}^{(s+4)})$. The other cases are similar. By definition, we have
\begin{eqnarray*}
& T_{k, \ell, 0}^{(s)} & =(3_{s}3_{s+4}\cdots 3_{s+4k-4})(2_{s+4k+1}2_{s+4k+3}\cdots 2_{s+4k+2\ell-3}2_{s+4k+2\ell-1}), \\
& T_{k-1, \ell-1, 0}^{(s+4)} & = ( 3_{s+4}\cdots 3_{s+4k-4})(2_{s+4k+1}2_{s+4k+3}\cdots 2_{s+4k+2\ell-3} ), \\
& M_1 & = T_{k, \ell, 0}^{(s)} T_{k-1, \ell-1, 0}^{(s+4)} A_{2, s+4k+2\ell-2}^{-1} \\
& & = T_{k, \ell, 0}^{(s)} T_{k-1, \ell-1, 0}^{(s+4)} 2_{s+4k+2\ell-1}^{-1}1_{s+4k+2\ell-2} 3_{s+4k+2\ell-2}.
\end{eqnarray*}
By $U_{q_2}(\hat{\mathfrak{sl}}_2)$ argument, it is clear that $n_1=M_1 A_{2, s+4k+2\ell-2}^{-1}$ is in $\chi_q(M_1)$.

If $n_1$ is in $\chi_q(T_{k, \ell, 0}^{(s)})  \chi_q( T_{k-1, \ell-1, 0}^{(s+4)} )$, then $T_{k, \ell, 0}^{(s)} A_{2, s+4k+2\ell-2}^{-1}$ is in $\chi_q(T_{k, \ell, 0}^{(s)})$ which is impossible by the FM algorithm for $\mathcal{T}_{k, \ell, 0}^{(s)}$. Similarly, $n_i\in \chi_q(M_i), i=2, \ldots, \ell-1$, but $n_2, \ldots, n_{\ell-1}$ are not in $\chi_q( T_{k, \ell, 0}^{(s)} )  \chi_q( T_{k-1, \ell-1, 0}^{(s+4)} )$.

By definition,
\begin{eqnarray*}
M_{\ell} & = & (3_{s}3_{s+4}\cdots 3_{s+4k-4})( 3_{s}3_{s+4}\cdots 3_{s+4k-8} ) 2_{s+4k-3} \times \\
& & \times (1_{s+4k} 1_{s+4k+2} 1_{s+4k+4} \cdots 1_{s+4k-2\ell-2} ) (3_{s+4k+2} 3_{s+4k+4} \cdots 3_{s+4k-2\ell-2} ).
\end{eqnarray*}
Let $U=\{ (2, aq^{s+4k}), (3, aq^{s+4k-2}) \} \subset I \times \mathbb{C}^{\times}$. Let $\mathcal{M}$ be the finite set consisting of the following monomials
\begin{align*}
m_0=M_{\ell}, \ m_1=m_0A_{3, s+4k-2}^{-1}, \ m_2=m_1A_{2, s+4k}^{-1}.
\end{align*}
It is clear that $\mathcal{M}$ satisfies the conditions in Theorem \ref{truncated}. Therefore
\begin{align*}
\text{trunc}_{M_{\ell} \mathcal{Q}_{U}^{-}}(\chi_q(M_{\ell}))=\sum_{m\in \mathcal{M}} m
\end{align*}
and hence $n_{\ell}=M_{\ell} A_{3, s+4k-2}^{-1} A_{2, s+4k}^{-1} $ is in $\chi_q(M_{\ell})$.

If $n_{\ell}$ is in $ \chi_q( T_{k, \ell, 0}^{(s)} )  \chi_q( T_{k-1, \ell-1, 0}^{(s+4)} ) $, then $ T_{k, \ell, 0}^{(s)} A_{3, s+4k-2}^{-1} A_{2, s+4k}^{-1} $ is in $\chi_q( T_{k, \ell, 0}^{(s)} )$ which is impossible by the FM algorithm for $ \mathcal{T}_{k, \ell, 0}^{(s)} $.

Similarly, for $i=\ell+1, \ldots, k+\ell-2$, we have $n_i\in \chi_q(M_i)$ and $n_i \not\in \chi_q( T_{k, \ell, 0}^{(s)} )  \chi_q( T_{k-1, \ell-1, 0}^{(s+4)} ) $.
\end{proof}

\section{Proof Proposition \ref{compute 1}} \label{prove compute}

Let $\mathcal{A}^{(s)}$ be a module in the system I. By (\ref{shift}), $\chi_q(\mathcal{A}^{(s)})$ is obtained from $\chi_q(\mathcal{A}^{(0)})$ by a shift of indices. For simplicity, we do not write the upper-subscripts "$(s)$" in the proof. The $q$-characters of Kirillov-Reshetikhin modules can be computed from $\chi_q(1_0), \chi_q(2_0), \chi_q(3_0)$, see \cite{Her06}.

It suffices to prove the following results.
\begin{enumerate}[(1)]
\item By using (\ref{relation 5 in the system I}), we can compute $\chi_q(T_{0, \ell, r})$, $r\in \{0, 1, 2\}$, $\ell \in \mathbb{Z}_{\geq 0}$.

\item By using (\ref{relation 2 in the system I}), (\ref{relation 3 in the system I}), and (\ref{relation 4 in the system I}), we can compute $\chi_q(T_{k, 0, m})$, $\chi_q(\tilde{T}_{k, 0, m})$, $\chi_q( S_{k, m} )$, $k, m \in \mathbb{Z}_{\geq 0}$.

\item By using (\ref{relation 1 in the system I}), (\ref{relation 6 in the system I})--(\ref{relation 9 in the system I}) and the $q$-characters computed in (1), (2), we can compute $\chi_q( R_{ k, 2\ell+i, \ell } )$, $\chi_q(T_{k, \ell, 0})$, $i=0, 1, 2$, $k, \ell \in \mathbb{Z}_{\geq 0}$.
\end{enumerate}

Proof of (1). We use induction on $\ell$, $r$. If $\ell \leq 1, r = 0$, then (1) is clearly true. Suppose that (1) is true for $\ell \leq \ell_1, r \leq r_1$, $\ell_1 \geq 1$, $r_1 \in \{0, 1, 2\}$. We will show that (1) is true for $\ell = \ell_1+1, r = r_1$ and $\ell = \ell_1, r = r_1+1$ respectively.

Suppose that $\ell = \ell_1+1, r = r_1$. If $r = 0$, then $T_{0, \ell, r}$ is a Kirillov-Reshetikhin modules and hence (1) is true. If $r = 1$, then by using (\ref{relation 5 in the system I}), $T_{0, \ell, r}$ can be computed by using the $q$-characters for Kirillov-Reshetikhin modules and induction hypothesis. If $r = 2$, then by using (\ref{relation 5 in the system I}), $T_{0, \ell, r}$ can be computed by using the $q$-characters for Kirillov-Reshetikhin modules, $\chi_q(0, k, 1)$, $k \in \mathbb{Z}_{\geq 0}$, and induction hypothesis. Similarly, we can show that (1) is true for $\ell = \ell_1, r = r_1+1$.

Proof of (2). It suffices to show the following result.

(2') By using (\ref{relation 2 in the system I}), (\ref{relation 3 in the system I}), and (\ref{relation 4 in the system I}), we can compute
\begin{align*}
& \chi_q(T_{k, 0, m}), \ k \leq k_1, m \leq m_1, \quad \chi_q(\tilde{T}_{m, 0, k}), \ k \leq k_1, m \leq m_1, \\
& \chi_q( S_{k, m} ), \ k \leq 2k_1, m\leq m_1-1, \quad \chi_q( S_{m, k} ), \ m \leq m_1-1, k \leq 2k_1.
\end{align*}
If $t_1 < 0$ or $t_2 < 0$, then we do not need to compute $S_{t_1, t_2}$.

We will prove (2') by using induction on $k_1$, $m_1$. If $k_1 = 1, m_1 = 1$, then (2') is clearly true. Suppose that (2') is true for $k_1=k_2, m_1=m_2$, $k_2, m_2 \geq 1$. We will show that (2') is true for $k_1 = k_2+1, m_1=m_2$ and $k_1=k_2,  m_1=m_2+1$ respectively.

Suppose that $k_1 = k_2+1, m_1 = m_2$. We only need to show that the following $q$-characters
\begin{align}
& \chi_q(S_{2k_2+1, m}), \ m \leq m_2 -1, \quad \chi_q(S_{2k_2+2, m}), \ m \leq m_2 -1, \quad \chi_q(T_{k_2+1, 0, m}), \ m \leq m_2, \\
& \chi_q(\tilde{T}_{m, 0, k_2 + 1}), \ m \leq m_2, \quad \chi_q( S_{m, 2k_2+1} ), \ m \leq m_2-1, \quad \chi_q( S_{m, 2k_2 + 2} ), \ m \leq m_2-1, \label{q-characters of T tilde}
\end{align}
can be computed. We compute the following $q$-characters
\begin{align*}
& \chi_q(S_{ 2k_2, 2 m - 1 }), \ m\leq \lfloor \frac{m_2-1}{2} \rfloor, \quad \chi_q(\tilde{T}_{2k_2+1, 0, m}), \ m\leq \lfloor \frac{m_2-1}{2} \rfloor, \\
& \chi_q(T_{k_2, 0, m}), \ m\leq m_2-1, \quad \chi_q(S_{2k_2+1, m}), \ m\leq m_2-1, \quad \chi_q(T_{k_2+1, 0, m}), \ m\leq m_2, \\
& \chi_q(S_{2k_2+1, m}), \ m\leq \lfloor \frac{m_2-1}{2} \rfloor, \quad \chi_q(\tilde{T}_{2k_2+2, 0, m}), \ m\leq \lfloor \frac{m_2-1}{2} \rfloor, \\
& \chi_q(S_{2k_2+2, m}), \ m\leq m_2 - 1,
\end{align*}
in the order as shown. At each step, we consider the module that we want to compute as a top module and use the corresponding relation in Theorem \ref{the system I} and known $q$-characters. Therefore $\chi_q(S_{2k_2+1, m})$, $m \leq m_2 - 1$, $\chi_q(S_{2k_2+2, m})$, $m\leq m_2 - 1$, $\chi_q(T_{k_1+1, 0, m})$, $m \leq m_2$, can be computed. The fact that the $q$-characters in (\ref{q-characters of T tilde}) can be computed can be proved similarly. Therefore (2') is true for $k_1 = k_2+1, m_1 = m_2$.

Similarly, we can show that (2') is true for $k_1 = k_2, m_1 = m_2+1$.

Proof of (3). It suffices to prove the following result.

(3') By using (\ref{relation 1 in the system I}), (\ref{relation 6 in the system I})--(\ref{relation 9 in the system I}) and the $q$-characters computed in (1), (2), we can compute the following $q$-characters:
\begin{align*}
& \chi_q( R_{ k , 2 \ell, \ell} ), \ k\leq 2k_1-3, \ell \leq \ell_1, \quad \chi_q( R_{ k , 2 \ell, \ell-1} ), \ k\leq 2k_1-1, \ell \leq \ell_1, \\
& \chi_q( R_{ k , 2 \ell-1, \ell-1} ), \ k\leq 2k_1-1, \ell \leq \ell_1, \quad \chi_q(T_{k, \ell, 0}), \ k \leq k_1, \ell \leq 2\ell_1.
\end{align*}

Let $k_2 = 1$, $\ell_2 = 1$. The $q$-characters of $R_{ k , i, 0} = T_{ k , i, 0}$, $i = 0, 1, 2$, are computed in (1). Therefore (3') is true in this case by using (\ref{relation 6 in the system I}) and the $q$-characters computed in (1).

Suppose that (3') is true for $k_1=k_2$, $\ell_1=\ell_2$, $k_2 \geq 1$, $\ell_2 \geq 0$. We will show that (3') is true for $k_1=k_2$, $\ell_1=\ell_2+1$ and $k_1=k_2+1$, $\ell_1=\ell_2$ respectively. Let $k_1=k_2$, $\ell_1=\ell_2+1$. We need to show that the following $q$-characters
\begin{align*}
& \chi_q( R_{ k , 2 \ell_2 + 2, \ell_2 + 1} ), \ k\leq 2k_2 - 1, \quad \chi_q( R_{ k , 2 \ell_2 + 2, \ell_2} ), \ k\leq 2k_2-1, \\
& \chi_q( R_{ k , 2 \ell_2+1, \ell_2} ), \ k\leq 2k_2-1, \quad \chi_q(T_{k, 2\ell_2+1, 0}), \ k \leq k_2, \quad \chi_q(T_{k, 2\ell_2+2, 0}), \ k \leq k_2,
\end{align*}
can be computed.
We compute the following $q$-characters
\begin{align*}
& \chi_q(T_{ k-2, 2 \ell_2, 0}), \ k\leq k_2, \quad \chi_q(R_{2k-3, 2\ell_2+1, \ell_2}), \ k\leq k_2, \\
& \chi_q(T_{ k-1, 2 \ell_2 + 1, 0}), \ k\leq k_2, \quad \chi_q(R_{2k-1, 2\ell_2+2, \ell_2}), \ k\leq k_2, \\
& \chi_q(R_{ 2k-1, 2 \ell_2 + 2, \ell_2+1}), \ k\leq k_2, \quad \chi_q(T_{k, 2\ell_2+2, 0}), \ k\leq k_2,
\end{align*}
in the order as shown. At each step, we consider the module that we want to compute as a top module and use the corresponding relation in Theorem \ref{the system I} and known $q$-characters. Therefore (3') is true for $k_1 = k_2, \ell_1 = \ell_2+1$.

Similarly, we can show that (3') is true for $k_1 = k_2+1, \ell_1 = \ell_2$.

\section{Conjectural character formulas} \label{conjectural formulas}

Every $U_q(\hat{\mathfrak{g}})$-module $V$ is also a $U_q(\mathfrak{g})$-module. We use $\res(V)$ to denote the $U_q(\mathfrak{g})$-module obtained by restricting $V$ to $U_q(\mathfrak{g})$.
We use the system I to compute the characters of the restrictions of modules in the system I. According to these computations, we obtain the following conjecture.
\begin{conjecture}\label{decomposition}
For $s\in\mathbb{Z}$, $k, \ell, m \in  \mathbb{Z}_{\geq 0}$, $\res(\mathcal{T}_{k, \ell, 0}^{(s)})$, $\res(\mathcal{T}_{k, 0, m}^{(s)})$, $\res(\mathcal{\tilde{T}}_{k, 0, m}^{(s)})$ are decomposed to direct sums of irreducible $U_q(\mathfrak{g})$-modules as follows.
\begin{align*}
\res(\mathcal{T}_{k, \ell, 0}^{(s)}) = \bigoplus_{j=0}^{i} \bigoplus_{i=0}^{\lfloor \frac{\ell}{2} \rfloor} V( (2i-2j)\omega_1 + (\ell-2i) \omega_2 + k \omega_3),
\end{align*}

\begin{align*}
\res(\mathcal{T}_{k, 0, m}^{(s)}) = \bigoplus_{i=0}^{\lfloor \frac{m}{2} \rfloor} V( (m-2i)\omega_1 + k \omega_3),
\end{align*}

\begin{align*}
\res(\mathcal{\tilde{T}}_{k, 0, m}^{(s)}) = \bigoplus_{i=0}^{\lfloor \frac{k}{2} \rfloor} V( (k-2i)\omega_1 + m \omega_3).
\end{align*}
\end{conjecture}
%Conjecture \ref{decomposition} can be proved if we know the decomposition formulas for $\res(\mathcal{T}_{0, \ell, r}^{(s)})$, $\res(\mathcal{S}_{k, \ell}^{(s)})$, $\res(\mathcal{R}_{k, 2\ell+j, \ell}^{(s)})$, $j=0, 1, 2$, $s\in\mathbb{Z}$, $k, \ell, m \in  \mathbb{Z}_{\geq 0}$.

\section*{Acknowledgements}
The author is very grateful to Evgeny Mukhin, Tomoki Nakanishi, and Charles A. S. Young for stimulating discussions. This research was partially supported by the National Natural Science Foundation of China (no. 11371177) and the Fundamental Research Funds for the Central Universities (Grant No. lzujbky-2012-12).

\end{document}